\documentclass[letterpaper, reqno, 11pt]{amsart}
\usepackage[english]{babel}
\usepackage{graphicx,color}
\usepackage[all]{xy}
\usepackage{amsfonts,mathrsfs}
\usepackage{amsmath, amsthm, amssymb}
\usepackage{enumerate, url}

\newcommand{\bbC}{\ensuremath{\mathbb{C}}}
\newcommand{\bbR}{\ensuremath{\mathbb{R}}}
\newcommand{\bbQ}{\ensuremath{\mathbb{Q}}}

\newcommand{\bbZ}{\ensuremath{\mathbb{Z}}}
\newcommand{\bbT}{\ensuremath{\mathbb{T}}}
\newcommand{\bbG}{\ensuremath{\mathbb{G}}}
\newcommand{\bbN}{\ensuremath{\mathbb{N}}}

\newcommand{\calO}{\ensuremath{\mathcal{O}}}
\newcommand{\calE}{\ensuremath{\mathcal{E}}}
\newcommand{\calS}{\ensuremath{\mathcal{S}}}

\newcommand{\calC}{\ensuremath{\mathcal{C}}}
\newcommand{\calD}{\ensuremath{\mathcal{D}}}

\newcommand{\inv}{^{-1}}
\newcommand{\sslash}{\mathbin{\mkern-2mu /\mkern-6mu/ \mkern-2mu}}

\DeclareMathOperator{\Id}{Id}

\DeclareMathOperator{\diam}{diam}

\DeclareMathOperator{\FSL}{FSL}

\DeclareMathOperator{\Hom}{Hom}
\DeclareMathOperator{\Aut}{Aut}

\DeclareMathOperator{\Out}{Out}

\DeclareMathOperator{\SL}{SL}

\DeclareMathOperator{\SO}{SO}

\DeclareMathOperator{\SU}{SU}

\DeclareMathOperator{\CGM}{\mathscr{C}(\Gamma\curvearrowright M)}
\DeclareMathOperator{\GGM}{\mathscr{G}(\Gamma\curvearrowright M)}
\DeclareMathOperator{\CLN}{\mathscr{C}(\Lambda\curvearrowright N)}
\DeclareMathOperator{\GLN}{\mathscr{G}(\Lambda\curvearrowright N)}
\DeclareMathOperator{\GGMtr}{\mathscr{G}_{t_0}(\Gamma\curvearrowright M)}
\DeclareMathOperator{\append}{append}

\DeclareMathOperator{\Isom}{Isom}
\DeclareMathOperator{\gr}{graph}

\DeclareMathOperator{\diag}{diag}
\numberwithin{equation}{section}

\newtheorem{thm}{Theorem}[section]
\newtheorem{prop}[thm]{Proposition}
\newtheorem{lem}[thm]{Lemma}
\newtheorem{cor}[thm]{Corollary}
\theoremstyle{definition}
\newtheorem{dfn}[thm]{Definition}

\newtheorem{rmk}[thm]{Remark}
\newtheorem{constr}[thm]{Construction}

\newtheorem{claim}[thm]{Claim}

\newtheorem{warning}[thm]{Warning}

\begin{document}

\title{Rigidity of warped cones and coarse geometry of expanders}

\author{David Fisher}
\author{Thang Nguyen}
\author{Wouter van Limbeek}

\date{\today}

\begin{abstract} We study the geometry of warped cones over free, minimal isometric group actions and related constructions of expander graphs. We prove a rigidity theorem for the coarse geometry of such warped cones: Namely, if a group has no abelian factors, then two such warped cones are quasi-isometric if and only if the actions are finite covers of conjugate actions. As a consequence, we produce continuous families of non-quasi-isometric expanders and superexpanders. The proof relies on the use of coarse topology for warped cones, such as a computation of their coarse fundamental groups.
\end{abstract}

\maketitle

\tableofcontents

\section{Introduction}
\label{sec:intro}

\subsection{Brief overview:} This paper can be viewed from two distinct points of view.  From one point of view, our primary contribution is proving the existence of wide classes of expanders and superexpanders which are geometrically different.  These results are discussed in Subsection \ref{sec:intro_expanders} of this introduction.  From another point of view, we prove a dramatic rigidity result for group actions.  Namely, we show that a certain geometric space associated to a group action, the warped cone of John Roe, is a complete conjugacy invariant for a large class of group actions. Our results on rigidity of group actions can be viewed as a dynamical analogue of {\em quasi-isometric rigidity} in geometric group theory and have other applications in the context of the coarse Baum-Connes conjecture. These results are discussed in Section \ref{sec:intro_cones}.  All the results on expanders and superexpanders are deduced from the rigidity results for group actions.  The reader can read the next two subsections of this paper in either order, depending on their interests.

\subsection{Warped Cones} \label{sec:intro_cones}

Let $\Gamma$ be a finitely generated group acting isometrically on a compact metric space $M$. For $t>0$, define the metric space $M_t :=(t M\times\Gamma)\slash\Gamma$, where $t M$ is a copy of $M$ scaled by a factor $t$ and $\Gamma$ is equipped with a fixed word metric. Define the \emph{geometric cone} $\GGM$ to be a metric space that assembles all $M_t$ as $\GGM=( M\times(0,\infty)\times\Gamma)\slash\Gamma$ where the metric on $M\times (0,\infty)$ is given by $t g + dt^2$ for some fixed Riemannian metric $g$ on $M$. For isometric $\Gamma$ actions  this definition coincides with Roe's warped cone over a group action \cite{MR2115671}, see Remark \ref{rmk:roe_cone}. At times it is better for applications to consider instead the metric space $\CGM$ that is simply the disjoint union of the $M_t$. We will study warped cones up to quasi-isometry.

\begin{dfn}
\label{dfn:qi}
Let $X$ and $Y$ be metric spaces. A map $f:X \rightarrow Y$ is an  $(L,C)$-quasi-isometry if there
exist $L>1$ and $C>0$ such that:
\begin{enumerate}
\item for every $a,b \in X$ we have $\frac{1}{L}d_X(a,b)-C \leq d_Y(f(a), f(b)) \leq Ld_X(a,b)+C$
\item for every $y \in Y$, there is $x \in X$ such that $d_Y(f(x),y)\leq C$.
\end{enumerate}
\end{dfn}

To study quasi-isometries of disjoint unions of level sets, we need to extend this notion slightly.
Given a collection of spaces $\{X_t\}_{t \in T}$ and $\{Y_s\}_{s \in S}$ we call them {\em $(L,C)$-quasi-isometric}
if there is a bijection $\phi: T \rightarrow S$ and $(L,C)$-quasi-isometries $f_t : X_t \rightarrow Y_{\phi(t)}$. As usual spaces or collections of spaces are {\em quasi-isometric} if there exist some $L>1$ and $C>0$ for which they
are $(L,C)$-quasi-isometric.

To state a version of our results we need to define what it means for group actions to be commensurable.  While this is similar to, and motivated by, the definition of commensurability for finitely generated groups, we warn the reader that actions of commensurable groups are not commensurable in this sense and in fact need not give rise to quasi-isometric warped cones!  The notion in this context is slightly more complicated than in the world of finitely generated groups.

%
%

\begin{dfn}
\label{dfn:commensurable}
Two group actions $\Gamma_0 \curvearrowright M_0$ and $\Gamma_1 \curvearrowright M_1$ are commensurable
if
\begin{itemize}
\item $M_i$ is a finite cover of $M_{i-1}$ for one of $i=0,1$  and
\item for the same index $i$ we have that $\Gamma_i$ is the group of lifts of $\Gamma_{i-1}$ to $M_i$.
\end{itemize}
\end{dfn}

\noindent In the definition the indices should be interpreted modulo $2$.  This is slightly awkward
but in the actual theorems it will be finite quotients of the original actions that are related,
not finite covers. This last notion corresponds most closely to the idea of two subgroups of a larger group having
intersection which has finite index in both, i.e. the extrinsic definition of commensurability. For the intrinisic definition of commensurablity, i.e. the analogue of having isomorphic finite index subgroups, it is useful to make explicit the notion of when two actions are isomorphic up to commensuration.  In particular, the regularity of the
commensurating map plays a role in applications.   

\begin{dfn}
\label{dfn:commensuration}
A {\em commensuration} of the actions $\Gamma\curvearrowright M$ and $\Lambda\curvearrowright N$ consists
of the following data:
	\begin{itemize}
		\item an action $\Gamma' \curvearrowright M'$ commensurable to $\Gamma \curvearrowright M$ and an action $\Lambda'\curvearrowright N'$ commensurable to $\Lambda\curvearrowright N$,
		\item a bi-Lipschitz map $f:M'\to N'$,%
        \item an isomorphism $\varphi:\Gamma'\to\Lambda'$,
    \end{itemize}   such that $f$ is $\varphi$-equivariant.
    If $M'$ and $N'$ are homogeneous spaces and $f$ is affine, then we call this data an {\em affine commensuration}.
\end{dfn}

\noindent Here a map between homogeneous spaces is \emph{affine} if it is compatible with the homogeneous structures, i.e. if it is the composition of a translation and a group automorphism (see Definition \ref{definition:affine}).

We state our main result first for geometric warped cones. Under an additional technical condition, we will obtain a similar result for non-geometric warped cones below.

\begin{thm}
\label{thm:main} Let $\Gamma$ and $\Lambda$ be finitely presented groups acting isometrically, freely and minimally on compact manifolds $M$ and $N$ (respectively). Assume that neither of $\Gamma$ and $\Lambda$ is commensurable to a group with a nontrivial free abelian factor.  Suppose that $\GGM$ is quasi-isometric to $\GLN$.  Then there is an affine commensuration of the two actions.
\end{thm}

\noindent For a minimal isometric $\Gamma$--action on $M$, the closure $G$ of $\Gamma$ in $\Isom(M)$ acts transitively on the compact manifold $M$ and minimality is preserved by commensurability. So our hypotheses imply that $M$ and $N$, as well as $M'$ and $N'$, are homogeneous and so the claim that the commensuration is affine is sensible.

We can also prove a result forcing an affine commensuration of actions even given only sequences $(t_n)$ and $(s_n)$ and quasi-isometries between the $t_n$--level set of $\CGM$ and the $s_n$--level set $\CLN$, but this requires some additional technical assumptions on acting groups.  The main advantage from our perspective of the geometric warped cone is that the maps on different level sets are in a natural sense {\em coarsely homotopic}.  The more general context is interesting because it is the avenue to our results about expanders and superexpanders.  For this more general result, we require an additional assumption on one of the acting groups, which we call {\em Property $\FSL$} and define in Subsection \ref{sec:stablenorm} below.  Groups with property $\FSL$ have only finitely many automorphisms satisfying any fixed constraint on distortion with respect to the stable length of group elements. Lattices in semisimple Lie groups have property $\FSL$. 

\begin{thm}Let $\Gamma$ and $\Lambda$ be finitely presented groups acting isometrically, freely and minimally on compact manifolds $M$ and $N$ (respectively). 	Assume that
	\begin{itemize}
		\item Neither of $\Gamma$ and $\Lambda$ is commensurable to a group with a nontrivial free abelian factor, and $\Gamma$ has Property $\FSL$, and
		\item There exist uniform quasi-isometries $f_n: M_{t_n}\to N_{s_n}$, where $t_n,s_n\to\infty$.
	\end{itemize}
Then there is an affine commensuration of the action of $\Gamma$ on $M$ and the action of $\Lambda$ on $N$.
	\label{thm:main_FSL}
	\end{thm}

\noindent We next state, somewhat informally, a corollary of Theorem \ref{thm:main_FSL}, our results for building many non-quasi-isometric expanders/group actions in Section \ref{sec:app_exp},  and the main results of Sawicki's recent paper on the coarse Baum-Connes conjecture for warped cones over group actions with spectral gap \cite{SawickiBC}. Sawicki's proof is based on ghost projections constructed in an earlier paper of Drutu and Nowak \cite{DrutuNowak}. We note that we use Theorem \ref{thm:main_FSL}, not Theorem \ref{thm:main}, since Sawicki is forced to pass to discrete versions of the cone for his proofs.

\begin{cor}
\label{cor:Sawicki}
There are continua of non-quasi-isometric spaces where the coarse Baum-Connes assembly map is not surjective.
\end{cor}
	
%

In the next subsection we discuss applications of our results to producing families of expanders that are not quasi-isometric.  In the few subsections following that one, we discuss the necessity of the hypotheses of these two theorems, some history of the work motivating our interest in this problem, and finally an outline of both our proofs and the paper.

\subsection{Expanders}\label{sec:intro_expanders}
Expander graphs play an important role in fields ranging from computer science to geometric topology to operator algebras.  An expander graph is a sequence of finite graphs of increasing size and bounded degree with strong spectral properties
that, for example, force embeddings of the graphs into Hilbert spaces to be highly distorted (see Section \ref{sec:gap-exp} for precise definitions). Recently, attention has
also been drawn to {\em superexpanders} where the spectral properties prevent good embeddings into broader classes of metric spaces. The first constructions of expanders were probabilistic, but so far there is no probabilistic construction of superexpanders.  More recently, attention has focused on constructions of (super)expanders with a wide range of geometric behavior. 
This began in  \cite[Theorem 9.1]{MR3210176} and
\cite[Theorem 5.76]{MR3114782} and has led to a rush  of recent results by Khukhro--Valette, Hume, Delabie--Khukro, De Laat--Vigolo and Sawicki \cite{MR3658731, MR3640615, DelabieKhukhro, deLaatVigolo, SawickiGap}.  Most of these papers
study the problem from the point of view of its intrinsic interest, but there are also intriguing connections to the
theory of computation in \cite{MR3352040}.

Explicit examples of expanders were first constructed by Margulis using congruence quotients of groups with property (T) \cite{MR0484767}. Using a strengthening of Property (T), Lafforgue showed that the same graphs are superexpanders \cite{MR2423763}. Group actions on homogeneous spaces are a geometric source of expanders: Indeed, Margulis also constructed expanders from the action of $\SL(n,\bbZ)$ on $\bbT^n$ for any $n\geq 2$ \cite{MR0484767}.  Margulis' construction is mostly carried out in the dual setting of $\SL(n,\bbZ)$ actions on $\bbZ^n$, while the explicit translation to actions on homogeneous spaces was first made by Gabber and Galil \cite{MR633542}.


 To see that group actions yield a robust source of expanders and superexpanders, we recall a discretiziation procedure that converts a group action into a sequence of graphs.  Given a group action on a metric space $(M, d)$, we can consider rescalings $M_t$ defined as $(M, td)$.  For a sequence $t_n\to\infty$, Vigolo discretizes the rescalings $M_{t_n}$ to get a sequence of approximating graphs $X_n$, and gives a criterion for $\{X_n\}_n$ to be an expander sequence \cite{Vigolo}. The approximating graph is obtained by partitioning each $M_t$ into sets of roughly equal diameter which correspond to vertices of the graph and then adding an edge between two elements in the partition if the translate of one by a generator intersects the other.
 Vigolo shows his discretizations are expanders if and only if the action of $\Gamma$ on $M$ has a spectral gap, a result extended to superexpanders by Sawicki \cite{SawickiGap}.  In fact, for the proofs of our theorem, we will need to pass from this sequence of graphs to one more closely tied to the geometry of $M$ as describe in \cite[Section 6]{Vigolo}.

When the underlying space is a manifold, we will refer to these as \emph{geometric expanders} or \emph{geometric superexpanders}.

There is a large literature producing actions with spectral gaps, so this construction provides a wealth of methods to produce expanders.  Since we will focus on minimal isometric group actions, we will need a compact group $K$ and a dense subgroup $\Delta <K$ with a spectral gap on $L^2(K)$.   For particular choices of $K$, many of these are constructed as lattices in a Lie group $G$ which have dense embeddings in a compact form $K$ of $G$, namely by Margulis and Sullivan using groups with property $(T)$ and by Drin'feld using more elaborate representation theory \cite{margulisspectral,sullivanspectral, drinfeldspectral}. In the late '90s, examples were found by more elementary methods by Gamburd--Jakobson--Sarnak in \cite{spectralgapconj} where the authors conjecture that a generic dense free subgroup of a compact simple group $K$ will have a spectral gap. In recent work towards that conjecture, Bourgain--Gamburd and Benoist--De Saxce show that any dense free subgroup of $K$ with algebraic entries has a spectral gap  \cite{MR2358056, MR2966656, MR3529116}.



Our contribution to the geometry of expanders will be in providing many examples of expanders that are not quasi-isometric.  While much of the work mentioned above is of this kind, so far applications to computer science require stronger results, namely that these expanders are, in an appropriate sense, an expander for random graphs \cite{MR3352040}.  It is possible that the geometric expanders considered here also satisfy this property, but a proof of this will require new ideas.  


We say two sequences of graphs $\{X_n\}$ and $\{Y_n\}$ are $(L,C)$--\emph{quasi-isometric} if there are maps $f_n:X_n \rightarrow Y_n$  that are $(L,C)$--quasi-isometries. (See Definition \ref{dfn:qi}.) Interested in the possible geometry of expanders, we define
two sequence of graphs to be $\{X_n\}$ and $\{Y_n\}$ to \emph{$(L,C)$-quasi-isometrically disjoint} if there are no $(L,C)$-quasi-isometric subsequences. (When discussing quasi-isometry of subsequences, we assume the obvious re-indexing.)  When the sequences are $(L,C)$-quasi-isometrically disjoint for all $L$ and $C$, we call them simply \emph{quasi-isometrically disjoint}. We call a collection $\Omega$ of sequences of graphs \emph{quasi-isometrically disjoint} if any pair of sequences in $\Omega$ is quasi-isometrically disjoint. From our results on group actions, we can produce many explicit families of quasi-isometrically disjoint expanders:

\begin{thm}
\label{thm:expanders}
Let $K$ be a compact Lie group and fix $r \geq 2$. Then free, rank $r$ subgroups of $K$ with spectral gap yield quasi-isometrically disjoint expanders unless they are commensurable. In particular, there exist explicit quasi-isometrically disjoint continuous families of expanders.
\end{thm}

To obtain families as in the theorem, we just note that from the existing literature we can produce many free groups with spectral gap: For example, set $K=\SU(n)$ with $n\geq 2$ and fix a free dense rank subgroup $\Lambda\subseteq \SU(n)$ of rank $r-1$ and with spectral gap. Then consider the family $\{\langle \Lambda, k\rangle\}_{k\in K}$. For generic $k$, we have $\Gamma:=\langle \Lambda, k\rangle \cong F_r$, and the spectral gap for $\Lambda$ implies that $\Gamma$ has spectral gap. For more details and other constructions, see Section \ref{sec:app_exp}.

We remark here that before the above result, the only known construction of uncountably many quasi-isometrically disjoint expanders is obtained by a fixing a single expander $X:=\{X_n\}_n$ with some geometric invariant, such as diameter, growing very rapidly, and considering the set of all subsequences of $X$. To show this produces uncountably many quasi-isometrically disjoint expanders, one uses that there are uncountably many sequences of natural numbers that pairwise have finite intersection, a fact that was pointed out to us by D. Sawicki.  As already pointed out by Vigolo, constructions from warped cones can lead to large families which cannot easily be distinguished by a condition on rapidly growing geometric invariants \cite{Vigolo}.  For example, in Theorem \ref{thm:expanders} and Theorem \ref{thm:superexpanders} below we can always produce examples where either number of vertices or the diameter of the $n$th graph is $n$.

%

From the point of view of coarse geometry a key property of expanders is that they do not embed
uniformly in any Hilbert space. By analogy, one defines superexpanders to be graphs that satisfy a certain Poincar\'{e} inequality that obstructs uniform embedding in some class $\calC$ of Banach spaces.  Here we will always consider the class $\calC$ of Banach spaces with nontrivial Rademacher type, which includes the class of uniformly convex Banach spaces. Superexpanders were first shown to exist by Lafforgue using strong property (T) in \cite{MR2423763}.  More examples were constructed by Mendel and Naor in \cite{MR3210176} and by Liao in \cite{MR3190138}. It is already clear from Lafforgue's first paper that his results, combined with Margulis' construction in \cite{margulisspectral}, produce geometric expanders. While there are in some sense many fewer examples of group actions with the stronger spectral properties required to produce geometric superexpanders, we are still able to construct {families} of these.
For the following theorem we let $G$ be a higher rank $p$-adic Lie group, $\Lambda <G$ a cocompact lattice with
a dense image homomorphism $\rho: \Lambda \to K$ where $K$ is a compact connected Lie group.  Given a free group $F_k$ with $k>2$, we let $\Gamma = \Lambda \times F_k$ which acts on $K$ by letting $\rho(\Lambda)$ act on the left, picking any inclusion $i: F_k \rightarrow K$  and letting $i(F_k)$ act on the right.

\begin{thm}
\label{thm:superexpanders}
Varying the inclusion $i$ above, the actions of $\Gamma$ on $K$ produce explicit quasi-isometrically disjoint families of superexpanders.
\end{thm}

\noindent In the proof of this theorem the $\Lambda$--action provides the spectral gap required to make the graphs superexpanders, while the free group provides us with the ability to vary the actions and geometry of the graphs.
Both Theorem \ref{thm:expanders} and \ref{thm:superexpanders} are consequences of Theorem \ref{thm:main_FSL} and some results on character varieties used to guarantee that the actions are free and minimal.

\subsection{History and motivations}
\label{sec:history}

Given a group $\Gamma$ and a sequence $\Gamma_n$ of finite index subgroups, the {\em box space} is the disjoint union of metric graphs $X_n = \Gamma/\Gamma_n$.  The key observation that begins the paper of Khukhro--Valette is that the Gromov--Hausdorff limit of $\{X_n\}_n$ is in fact $\Gamma$, so that if there is a quasi-isometry between box spaces of $\Gamma$ and $\Lambda$, then $\Gamma$ and $\Lambda$ are quasi-isometric and one has access to the broad literature on quasi-isometries of groups.  The proof of De Laat--Vigolo follows a similar line but replaces box spaces by warped cones.  They study warped cones over actions on manifolds and show that any sequence of level sets for $\CGM$ has Gromov--Hausdorff limit $\Gamma \times \bbR^{\dim(M)}$, a result also proven by Sawicki in \cite{Sawicki}.  In both cases these results allow one to distinguish large classes of expanders by referring to the literature on quasi-isometric rigidity of groups.

More recently, Delabie--Khukhro have shown that quasi-isometries of box spaces are in fact much more restricted than just quasi-isometries of the ambient group \cite{DelabieKhukhro}.  Essentially Delabie--Khukhro use that the quasi-isometry makes the box spaces coarsely homotopic and show that this forces the ambient groups to be commensurable and the sequences of subgroups to be isomorphic using the notion of a coarse fundamental group.  These results are the natural analogue of ours in the setting of box spaces, but do not require any assumptions on the groups. The fact that there is some obstruction to a result like ours in terms of the acting group is already noticed explicitly by Sawicki, who proves a conditional result about quasi-isometries of warped cones induced by a single good map $f: M \rightarrow N$ \cite{Sawicki}. While we do not use this result of Sawicki's here, from this perspective, a key technical difficulty in our work is producing any such map.

We are also motivated by the natural analogy between classifying finitely generated groups up to quasi-isometry and classifying the warped cones of group actions up to quasi-isometry.  The new classification questions that arise seems to connect to numerous other fields of mathematics including computer science, dynamics, geometric group theory and operator algebras.  We remark here that much as it is for finitely generated groups, {\em quasi-isometry} is the right notion of equivalence in the category of warped cones.  First, changing the generating set of the acting group changes the warped cone metric but the metrics remain quasi-isometric.  Second, choosing a different but bi-Lipschitz metric on $M$ also changes the warped cone metric, but again only up to quasi-isometry.  This second observation implies that all Riemannian metrics on a fixed manifold will give rise to quasi-isometric warped cones.  In the context of actions on totally disconnected spaces, Sawicki shows that non-bi-Lipschitz metrics can give rise to non-quasi-isometric warped cones for the same group action.  One can probably also produce examples of this in the setting of actions on manifolds simply by taking a Riemannian metric and {\em snowflaking} it.

\subsection{Hypothesis on group actions}
\label{sec:hypo}

In both Theorems \ref{thm:main} and \ref{thm:main_FSL}, the assumption on absence of abelian factors of $\Gamma$ and $\Lambda$ is necessary: In \cite{MR2252891}, Kim shows that circle rotations that are related by integral fractional linear transformations give rise to quasi-isometric geometric warped cones.  Since rotation number is a conjugacy invariant of rotations, this provides many non-conjugate $\bbZ$--actions on $S^1$ that have quasi-isometric geometric warped cones and the rotation numbers can be sufficiently different that the problem is not resolved by commensuration of the actions.  In an appendix to a recent paper by Sawicki, Kielak--Sawicki give an even more remarkable construction showing that there are rotations on $S^1$ whose warped cone is quasi-isometric to the warped cone of the trivial action on $\bbT^2$ \cite{Sawicki} and also produce a variety of other examples of this flavor, including ones between minimal isometric actions. The key insight there is to see that one can shift geometry of the acting group to geometry of the ambient space, provided one takes sequences of level sets that grow at different rates and the geometry of the space and group factors are sufficiently similar.  These more exotic examples do not give rise to quasi-isometries of geometric warped cones.

We do not currently know if Property $\FSL$ is necessary in Theorem \ref{thm:main_FSL} or if all finitely generated subgroups of compact Lie groups have Property $\FSL$.  We have also not explored the degree to which variants of the theorems might be true for actions that are not isometric, which is an intriguing direction for future research.

\subsection{Basic intuition and outline of the proofs and the paper}

The basic intuition behind our theorem is roughly that if the quasi-isometries between the cones respected a distinction between directions related to paths in the manifold and directions related to paths in the group, then the quasi-isometry would appear to need to recognize the orbit structure of the action. This is because points $x, y$ in $M$ are in the same $\Gamma$--orbit if and only if the distance between $x$ and $y$ in level sets of $\CGM$ remains bounded.  In addition if the quasi-isometry respects the distinction between directions related to paths in the manifold and directions related to paths in the group, one should be able to rescale the quasi-isometries along the manifold direction to have them converge to a bi-Lipschitz map.  Together these two intuitions would produce an orbit equivalence between actions, which can be seen to be a conjugacy by results of the first author and Whyte \cite{MR2033836}.  The actual proof does not follow exactly this line and as mentioned above, the examples of Kim and Kielak--Sawicki show that the basic intuition is wrong for at least some examples. Their intriguing examples show that one cannot, in general, hope that quasi-isometries of warped cones respect the distinction between directions related to paths in the manifold and paths in the group.  Our theorems and proofs do give strong evidence that in the case
of minimal isometric actions, their examples capture more or less all the ways in which this mixing of space and orbit can occur for quasi-isometries of warped cones.

Before beginning our proof, Section \ref{sec:prelim} recalls definitions of warped cones and coarse fundamental groups.   The first step of our proof is to show that the coarse fundamental group of a warped cone $\CGM$ is the group $\widetilde{\Gamma}$ of lifts of the $\Gamma$--action to the universal cover $\widetilde{M}$ of $M$. A similar computation was obtained independently and simultaneously by Vigolo \cite{vigolo_coarsepi1} (see Remark \ref{rmk:vigolo_coarsepi1} for further discussion). As quasi-isometries necessarily induce isomorphisms of coarse
fundamental groups, we get an isomorphism between $\widetilde{\Gamma}$ and $\widetilde{\Lambda}$ induced by the quasi-isometry of warped cones.  In Section \ref{sec:no_obstruct}, we use this to define an {\em obstruction map} that captures when a quasi-isometry of warped cones mixes group directions with manifold directions.  The obstruction map is a composition

$$\pi_1(M) \rightarrow \widetilde{\Gamma} \overset{\cong}{\longrightarrow} \widetilde{\Lambda} \rightarrow \Lambda$$
\noindent where the first and last arrow come from the structure of groups of lifts and the isomorphism in the middle arises from our quasi-isometry.  We show that after passing to commensurable actions, our hypothesis that $\Gamma$ and $\Lambda$ have no abelian direct factors forces the obstruction map to vanish. At this point we have isomorphisms $\pi_1(M) \cong \pi_1(N)$ and $\widehat{\varphi}:\Gamma \rightarrow \Lambda$.  The existence of $\widehat{\varphi}$ is already a major development, since we now have an identification of acting groups and so it is sensible to ask if the actions are conjugate.

The outline is actually a bit more complicated at this step.  If we have a quasi-isometry of geometric warped cones, it is not hard to see purely from coarse topology that $\widehat{\varphi}$ actually exists and is independent of $t$ (see Section \ref{sec:geomcone_stab}).  In the context of Theorem \ref{thm:main_FSL}, we actually have a sequence $\widehat{\varphi}_n$ of isomorphisms and a key step is to argue that this sequence of quasi-isometries stabilize along a subsequence.  We actually need only that they stabilize up to conjugation, i.e. that $\widehat{\varphi}_n$ stabilize in $\Out(\Gamma)$ rather than in $\Aut(\Gamma)$. This is trivial if say $\Out(\Gamma)$ is finite, but in general requires an argument.  In Section \ref{sec:stab}, we show that the group automorphisms that arise in our context are {\em Lipschitz for the stable length}, and we introduce Property $\FSL$ which essentially says that given a constant $D$ there are only finitely many elements of $\Out(\Gamma)$ represented by automorphisms that are $D$--Lipschitz for stable length.

The next step after this is also a matter that only arises in the setting of Theorem \ref{thm:main_FSL}.  The reader
only interested in Theorem \ref{thm:main} can skip all but the first and last subsections of Section \ref{sec:stab} and all of Section \ref{sec:force_linear}.    In Section \ref{sec:force_linear}, we show that vanishing of the obstruction map in the context of Theorem \ref{thm:main_FSL} forces $t_n$ and $s_n$ to be linearly related.  Once one knows this,
it is easy to see that one can assume $t_n=s_n$.  We note here that it is easy to compute the obstruction map
for the examples of Kim and Kielak--Sawicki and see that it does not vanish.

The main step of our proof, completed in Section \ref{sec:semiconj}, is to show that one can now in fact construct a semiconjugacy between the actions.   To do so, we modify the map of levels sets given by the quasi-isometry by a monodromy correction, which measures the extent to which the quasi-isometry mixes manifold and orbit directions. We show these modified maps are {\em coarsely Lipschitz} maps between the manifolds.  The estimate on the constants is such that the resulting maps converge to a Lipschitz map as $t \rightarrow \infty$. The stabilization up to conjugacy of the isomorphisms $\varphi_n$, allows us to further adjust the quasi-isometries of level sets before passing to the limit, so that in the limit we produce a Lipschitz semiconjugacy.

The final issue is to show that the semi-conjugacy we construct is actually an affine conjugacy, which is proved using a variant of a standard argument from homogeneous dynamics in Section \ref{sec:conj}.

Lastly in Section \ref{sec:app_exp}, we prove Theorems \ref{thm:expanders} and \ref{thm:superexpanders} from Theorem \ref{thm:main_FSL}.  The key idea is to produce continua of group actions that are not affinely conjugate but that have the necessary spectral gaps.  To produce expanders, it suffices to consider actions of  free group $F_k$ where we fix the action on $F_{k-1}$ to be one with a spectral gap and vary the last generator as described after Theorem \ref{thm:expanders}.  To produce superexpanders, where we need a spectral gap with more general Banach coefficients, we consider products $\Delta \times F_n$ acting on a compact Lie group $K$ where $\Delta$ acts by left-translations and $F_n$ acts by right-translations. We guarantee spectral gap for this action by choosing the $\Delta$--action so that it already has spectral gap, and this will be inherited by the product.  We remark that to think of this as an affine action requires
thinking of $K$ as $(K \times K)/ \diag(K)$ where $\diag(K)$ is the image of the diagonal embedding of $K$.  The existence of continua of non-conjugate free minimal actions follows from results on character varieties proved using Borel's theorem that word maps are dominant, see Subsection \ref{sec:free_generic}.

\subsection*{Acknowledgments} The authors thank Michael Larsen and Nick Miller for useful conversations on
representation varieties and word maps. DF thanks Piotr Nowak for introducing him to this circle of ideas and thanks Assaf Naor and Mikeal de la Salle for interesting conversations about superexpanders and
strong property $(T)$. WvL is grateful to Federico Vigolo for explaining the connection between expanders and warped cones over group actions at the Newton Institute in May 2017, which aroused his interest in this topic.  We thank Emmanuel Breuillard for his explanation of the proof of Lemma \ref{lemma:charactervariety}. We also thank Vigolo and Damian Sawicki for helpful comments on a preliminary version of this paper. DF was partially supported by NSF Grants DMS--1308291 and DMS--1607041.   WvL and TN thank the Indiana University Math Department for support during a visit there and TN thanks the Michigan Math Department and NSF Grant DMS--1607260 for support during a visit there.

\section{Preliminary remarks}
\label{sec:prelim}

\subsection{Generalities for warped cones} \label{sec:warp_cone}

Let a finitely generated group $\Gamma$ act isometrically on a closed Riemannian manifold $(M,g)$. Let $S$ be a symmetric generating set for $\Gamma$, so that $\Gamma$ is equipped with a word metric. The \emph{geometric cone} of $\Gamma\curvearrowright M$ is the metric space
	$$\GGM:=(M\times (0,\infty)\times\Gamma)\slash\Gamma,$$
where $M\times(0,\infty)$ is equipped with the warped product metric $tg+dt^2$. We let $tM$ denote the manifold $M$ with the metric $tg$, and write $M_t:=(tM\times\Gamma)\slash\Gamma$ for the \emph{level set of the geometric cone at time} $t$. We write $d_t^\Gamma$ for the metric on $M_t$. It is immediate that $M_t$ can also be described as the metric space with underlying set $M$ and $d_t^\Gamma$ defined as follows:
\begin{lem} Let $x,y\in M$. Then
	$$d_t^\Gamma(x,y)=\inf_{\gamma\in\Gamma} (t \, d_M(x,\gamma y) + \|\gamma\|).$$
\label{lem:warp_dist}
\end{lem}

\begin{rmk} For clarity, we mention that our definition of the metric on the (geometric) warped cone differs from
Roe's.  For isometric actions, it is well known that Roe's definition is equivalent to the one given by the
formula in the lemma above and so to ours.
\label{rmk:roe_cone}
\end{rmk}

From the above expression, it is clear the large-scale geometry of these level sets is invariant under taking finite quotients or extensions of the action:
\begin{lem} Let $F\subseteq \Gamma$ be a finite normal subgroup, and write $\Gamma':=\Gamma\slash F$ and $M':=M\slash F$. Write $M_t$ (resp. $M_t'$) for the level set at time $t$ of the cone over $\Gamma\curvearrowright M$ (resp. $\Gamma'\curvearrowright M'$). Then $M_t$ is $(1,B)$--quasi-isometric to $M_t'$, where $B$ is the maximal word length in $\Gamma$ of an element of $F$.
\label{lem:cone_fin_quot}
\end{lem}
\begin{proof} Write $\pi:M_t\to M_t'$ for the natural quotient map. It is easy to see that $\pi$ is distance non-increasing, i.e. $1$--Lipschitz. For the lower bound, let $C$ be the maximal word length of an element of the finite group $F$. Further for any $\gamma\in\Gamma$, let $\gamma'$ denote its image in $\Gamma'$. Then we have $\|\gamma\|\leq \|\gamma'\| + C$, which easily implies the other bound.\end{proof}
We also define the \emph{full warped cone} $\CGM:=\sqcup_t \, M_t$ and more generally, for any sequence $t_n\to\infty$, the associated warped cone $\mathscr{C}_{(t_n)}(\Gamma\curvearrowright M):=\sqcup_n \, M_{t_n}$. For most of the paper, and for applications to expander graphs, we need to consider quasi-isometries of warped cones associated to a sequence $(t_n)$, i.e. a sequence of uniform quasi-isometries $f_n: M_{t_n}\to N_{s_n}$. Note that in general it need not be true that $t_n\asymp s_n$ -- indeed, we already alluded to an example of Kielak--Sawicki that produces a quasi-isometry between the warped cone over a suitable rotation on $S^1$ for a quadratically growing sequence, and the warped cone over the trivial action on $T^2$ for a linear sequence. However, see Section \ref{sec:force_linear}, where we show that $t_n\asymp s_n$ under suitable hypotheses.

A map $f:X\to Y$ of metric spaces is called $(L,C)$--\emph{coarsely Lipschitz} if for any $x_1,x_2\in X$, we have
	$$d_Y(f(x_1),f(x_2))\leq L d_X(x_1,x_2)+C.$$
In general, a quasi-isometry of geometric cones yields uniformly coarsely Lipschitz maps of levels.
\begin{lem} Let $f:\GGM\to \GLN$ be a quasi-isometry. For any $t>0$, define $f_t:M_t\to N_t$ as the composition
	$$f_t: M_t\hookrightarrow \GGM \overset{f}{\longrightarrow} \GLN \twoheadrightarrow N_t$$
where the first map is the inclusion of a level set and the last map is the projection to a level set. Then there exists $t_0>0$ such that for $t>t_0$, the maps $f_t$ are uniformly coarsely Lipschitz.
	\label{lem:global_to_local}
\end{lem}
\begin{proof} It is clear the inclusion of a level set is $1$-Lipschitz and $f$ is $(L,C)$-coarsely Lipschitz. Therefore to show the composition is uniformly coarsely Lipschitz, we need to bound the distortion coming from the projection to a level set. More precisely, let $p_2: \GGM \to (0,+\infty)$ denote the projection of the warped cone on the second factor, and likewise define $q_2:\GLN\to(0,\infty)$. We claim that $f$ coarsely preserves the second factor, namely that there exists $K\geq 1$ and $B>0$ such that for all $x\in\GGM$, we have
	\begin{equation} K\inv p_2(x)-B\leq q_2(f(x))\leq K p_2(x)+B. \label{eq:ray_dir} \end{equation}
Indeed, consider a basepoint $x_0\in\GGM$ with $p_2(x_0)=1$. It is easy to obtain the estimates of Equation \eqref{eq:ray_dir} from the observation that $p_2$ and $q_2$ are distance non-increasing, and further
	$$d(x,x_0)\leq |p_2(x)-p_2(x_0)|+\diam(M),$$
and likewise for $d(f(x),f(x_0))$.

From Equation \eqref{eq:ray_dir}, it is easy to see that the projection of $f(M_t)$ to the level set $N_t$ increases distances by a factor of at most $(K\inv t-B)\inv t$, which is bounded above by $(2K)\inv$ for $t\gg 1$. \end{proof}
%



\subsection{Coarse homotopy and fundamental group} \label{sec:coarse_htpy}
We briefly recall the definition of and basic facts about coarse fundamental groups. For more information, see e.g. \cite{MR3262192, DelabieKhukhro}.

\begin{dfn} Let $X$ be a metric space and $r>0$. An \emph{$r$-path} is a sequence of points $\alpha=(p_0,\dots,p_\ell)$ where $d(p_i,p_{i+1})<r$. An \emph{$r$-loop} is an $r$-path with $p_0=p_\ell$. The \emph{endpoint append} of an $r$-path $\alpha$ is the $r$-path
	$$\append(\alpha)=(p_0,\dots,p_\ell,p_\ell).$$
If $\beta=\append(\alpha)$, we say $\alpha$ is the \emph{endpoint deletion} of $\beta$. Two $r$-paths $\alpha=(p_0,\dots,p_\ell)$ and $\beta=(q_0,\dots,q_n)$ are \emph{$r$-close} if either
	\begin{enumerate}[(a)]
		\item $\beta$ is the endpoint append or deletion of $\alpha$, or
		\item $\alpha$ and $\beta$ have the same length (i.e. $\ell=n$) and $d(p_i,q_i)<r$ for all $i$.
	\end{enumerate}
An \emph{$r$-homotopy} based at $p_0$ between two $r$-loops $\alpha$ and $\beta$ based at $p_0$ is a sequence of $r$-loops based at $p_0$
	$$\alpha=c_0, c_1,\dots, c_{k-1}, c_k=\beta$$
where $c_i, c_{i+1}$ are $r$-close for all $i$.
\label{dfn:rpath}
\end{dfn}

We define the coarse fundamental group:
\begin{dfn} Let $X$ be a metric space and $r>0$ and $p_0\in X$. The \emph{coarse fundamental group of $X$ based at $x_0$ and of scale $r$} is the group $\pi_1(X, x_0; r)$ whose elements are $r$-homotopy classes of $r$-loops based at $x_0$, with multiplication given by concatenation.
\label{dfn:coarse_pi1}
\end{dfn}
Let us first discuss dependence on the basepoint $x_0$. The situation is completely parallel to the one for (continuous) fundamental groups. Namely, we say a metric space is \emph{connected on the scale} $r$ if there is an $r$-path between any two of its points. If a space is connected on the scale $r$, then the isomorphism type of $\pi_1(X, x_0; r)$ is independent of $x_0$. In this case we refer to this group as $\pi_1(X; r)$. Any two possible identifications differ by an inner automorphism.

Let us now discuss basic functoriality properties of coarse fundamental group of a metric space $X$. First, we can compare the coarse fundamental group at different scales $r_1<r_2$. Namely, since any $r_1$-path is an $r_2$-path, we have an obvious \emph{change of scale}--map
	$$\pi_1(X, x_0; r_1)\to \pi_1(X, x_0; r_2).$$
If $c$ is a (continuous) path in $X$, we say an $r$-path $\alpha=(p_0,\dots,p_\ell)$ is a \emph{discretization of} $c$ \emph{on the scale} $r$ if there is a partition
	$$0=t_0<t_1<\dots<t_\ell=1$$
of $[0,1]$ such that $p_i=c(t_i)$ and $c([t_i,t_{i+1}])\subseteq B_X(p_i; r)$ for every $i$. The $r$-homotopy class of a discretization is invariant under refinement under the partition and hence does not depend on the partition chosen. It is not hard to check that (continuous) homotopies of loops based at $x_0$ give $r$--homotopic discretizations. Therefore we have a well-defined discretization map
	$$\pi_1(X)\to \pi_1(X;r).$$
An $(L,C)$--\emph{coarsely Lipschitz} map defines an obvious morphism on coarse fundamental groups
	$$f_\ast:\pi_1(X,x_0;r)\to \pi_1(Y, f(x_0);Lr+C).$$
These behave well with respect to composition, namely if $f:X\to Y$ is $(L_f,C_f)$--coarsely Lipschitz and $g:Y\to Z$ is $(L_g,C_g)$--coarsely Lipschitz, then $g\circ f$ is $(L_g L_f, L_g C_f + C_g)$--coarsely Lipschitz and $(g\circ f)_\ast = g_\ast\circ f_\ast$.
	
\subsection{Technical coarse homotopy lemma} Aside from endpoint appends and deletions, we can perform any coarse homotopy by making only two changes at the time, and these are in consecutive positions:

\begin{lem} Suppose $\alpha, \beta$ are $r$-close $r$-loops based at $p_0$ of the same length. Then there is an $r$-homotopy (possibly through curves of longer length) from $\alpha$ to $\beta$ such that consecutive paths differ in at most two points, and these are consecutive (we allow endpoint appends and deletions).
\label{lem:baby_step_htpy}
\end{lem}

\begin{proof} Write $\alpha=(p_0,\dots,p_{\ell+1})$ where $p_0=p_{\ell+1}$ and likewise write $\beta=(p_0,q_1,\dots,q_\ell,p_0)$. Since $\alpha$ and $\beta$ are $r$-close, we know $d(p_i,q_i)<r$ for all $1\leq i\leq l$.

It is easy to see that there is an $r$-homotopy as desired (i.e. making only local changes) between $\alpha$ and the loop that prepends the point $p_0$ to $\alpha$, i.e. the loop
	$$\text{prepend}(\alpha):=(p_0, p_0, p_1, \dots,p_{\ell+1}).$$
This is $r$-close to the loop obtained by replacing the second entry by $q_1$, i.e.
	$$\eta_1:=(p_0, q_1,p_1,\dots,p_{\ell+1}).$$
Now we take a homotopy with only local changes to the loop which is exactly the same except repeating the point $p_2$ twice:
	$$(p_0, q_1, p_1, p_2, p_2, p_3, \dots, p_{\ell+1}).$$
Now we make a `synchronized jump' where we change the third entry from $p_1$ to $q_1$ and the fourth entry from $p_2$ to $q_2$ and end up with the loop
	$$\eta_2:=(p_0,q_1, q_1, q_2, p_2, p_3,\dots,p_{\ell+1}).$$
See Figure \ref{fig:baby-step-htpy} for the idea.
	\begin{figure}[h]
		\includegraphics[height=4 cm]{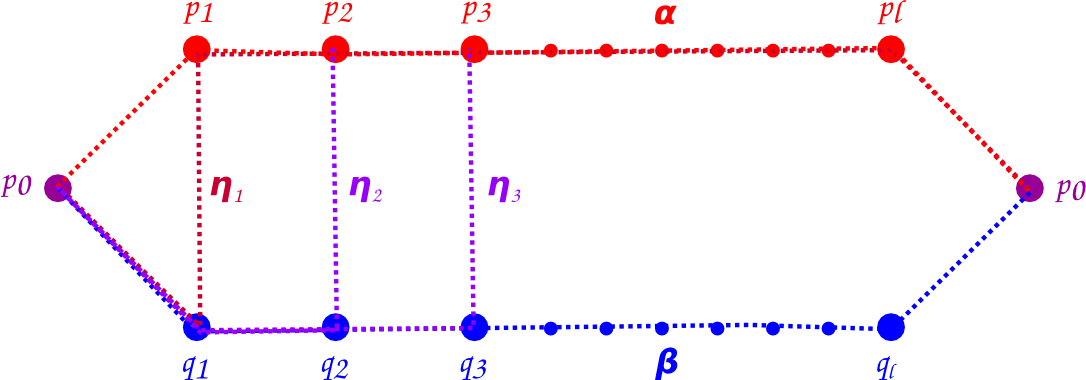}
		\caption{Homotopy from $\alpha$ (path along the top) to $\beta$ (path along the bottom) by only making local changes.}
		\label{fig:baby-step-htpy}
	\end{figure}
Now we move the repeated entry of $q_1$ over so that $p_3$ is repeated twice, and we do another synchronized jump. We keep doing synchronized jumps until we changed all points $p_i$ to points $q_i$. Finally we move any repeated points so that they are at the end, and after endpoint deletions we obtain $\beta$.\end{proof}

\subsection{Additional notation and terminology} Let $\alpha=(p_i)_{0\leq i\leq n}$ be an $r$-path in a compact metric space $X$. The \emph{length} $\ell(\alpha)$ of $\alpha$ is defined by
	$$\ell(\alpha):=\sum_{0\leq i<n} d(p_i,p_{i+1}).$$
We denote by $\overline{\alpha}=(p_{n-i})_{0\leq i\leq n}$ the \emph{reverse} of $\alpha$. If $\beta=(q_i)_{0\leq i\leq m}$ is another $r$-path in $X$ with $d(p_n,q_0)\leq r$, then
	$$\alpha\ast\beta:=(p_0,\dots,p_n,q_0,\dots,q_m)$$
denotes the \emph{concatenation} of $\alpha$ and $\beta$.

%
%
%
%
%

\section{Coarse fundamental group of level sets of warped cones}
\label{sec:coarsepi1}

\subsection{Statement of result} The goal of this section is to compute the coarse fundamental group of the level set $M_t$ of $\CGM$ at time $t$. A heuristic comes from the analogy between coarse algebraic topology on $M_t=(t M\times\Gamma)\slash\Gamma$ and (continuous) algebraic topology on $E:=(M\times\widetilde{X})\slash\Gamma$, whenever $X=\widetilde{X}\slash\Gamma$ is an aspherical manifold with fundamental group $\Gamma$. After lifting the $\Gamma$--action on $M$ to an action of a group $\widetilde{\Gamma}$ on the universal cover $\widetilde{M}$ of $M$, we have $E=(\widetilde{M}\times\widetilde{X})\slash\widetilde{\Gamma}$, so we see that $\pi_1(E)\cong\widetilde{\Gamma}$ is the group of lifts of $\Gamma$ to $\widetilde{M}$. We prove the same result holds in the coarse setting:
	\begin{prop} Let $\Gamma$ be a finitely presented group acting minimally, isometrically, and freely on a manifold $M$. Let $\widetilde{\Gamma}$ be the group of lifts of $\Gamma$ to the universal cover $\widetilde{M}$ of $M$. Then for every $t\gg r\gg 1$ and $p_0\in M$, we have
	\begin{equation*} \pi_1(M_t, p_0; r)\cong \widetilde{\Gamma}.\end{equation*}
	\label{prop:coarsepi1}
	\end{prop}
	\begin{rmk} A computation of coarse fundamental groups of warped cones was independently and simultaneously obtained by Vigolo \cite{vigolo_coarsepi1}. Note that the setting there is 	much more general than the one we consider here, although the final answer is presented in a form that is not adapted to the relationship between $\Gamma, \pi_1(M)$ and $\pi_1(M;r)$ that we need at present.
	\label{rmk:vigolo_coarsepi1}
	\end{rmk}
The stabilization of the coarse fundamental group $\pi_1(M_t, p_0;r)$ as $t\to\infty$ easily allows for the following computation of the fundamental group of geometric cones. However, the full geometric cone is coarsely contractible (because the manifold has been coned off, at least on the scale $r$). To counteract this, we introduce a truncated geometric cone: For $t_0>0$, we write $\GGMtr:=(M\times[t_0,\infty)\times\Gamma)\slash\Gamma$ for the geometric cone truncated at level $t_0$. It is easy to see the projection $\GGMtr\to M_{t_0}$ is a coarse deformation retract, so we have:
	\begin{cor} Let $\Gamma$ be a finitely presented group acting minimally, isometrically, and freely on a manifold $M$. Then for $t>t_0\gg r\gg 1$, the inclusion $M_{t_0}\hookrightarrow \GGMtr$ induces an isomorphism on coarse fundamental groups at scale $r$. In particular, we have
	$\pi_1(\GGMtr, [p_0,t_0]; r)\cong \widetilde{\Gamma}.$
	\label{cor:coarsepi1_geomcone}
	\end{cor}
We provide a brief outline of the proof of Proposition \ref{prop:coarsepi1}: In Section \ref{sec:contrib_act} we associate to every coarse loop an element of $\gamma$, based on the observation that given two sufficiently close points $x,y$ in $M_t$, there is a unique element of $\Gamma$ that translates $x$ to a point near $y$ in $M$. In this way we associate to a coarse loop finitely many elements of $\Gamma$. By taking the product we obtain a map $\pi_1(M, p_0; r)\to\Gamma$. In Section \ref{sec:contrib_mnfd} we show that discretization of continuous loops gives a map $\pi_1(M,p_0)\to \pi_1(M,p_0;r)$ that is an embedding for $t\gg r\gg 1$. In Section \ref{sec:canonical}, we introduce a technical tool that decomposes an arbitrary coarse path (up to coarse homotopy) into an orbital and spatial segment. Next, in Section \ref{sec:mnfd_gp_rel}, we use this to study the relationship between the contributions of the action and the manifold to the coarse fundamental group, and show $\pi_1(M,p_0;r)$ is an extension of $\Gamma$ by $\pi_1(M)$. Finally in Section \ref{sec:ext_lift} we show that this extension is in fact $\widetilde{\Gamma}$.

\subsection{The contribution of the action} \label{sec:contrib_act} We start by associating an element of $\Gamma$ to every coarse loop in $M_t$. For $R>0$, set
	$$\delta_\Gamma(R):=\min_{\|\gamma\|\leq R} \, \min_{x\in M} d(x,\gamma x).$$
\begin{constr} Let $t>\frac{2r}{\delta_\Gamma(2r)}$. Let $\alpha=(p_0,p_1,\dots,p_\ell,p_0)$ be an $r$-path based at $p_0$. Since
	$$d_t^\Gamma(p_i,p_{i+1})=\inf_\gamma \left(\|\gamma\| + t \, d_M(\gamma p_i, p_{i+1})\right)<r,$$
we see there exists some $\gamma\in B_\Gamma(r)$ with $d_M(\gamma p_i, p_{i+1})<\frac{r}{t}$. Further this element $\gamma$ is unique, because if $\eta$ is another element with this property, then $\gamma^{-1}\eta$ is an element of length at most $2r$ that displaces $p_i$ by at most $\frac{2r}{t}$. Since $t>\frac{2r}{\delta_\Gamma(2r)}$, this forces $\gamma=\eta$, as desired.

Hence we can associate to the $r$-path $\alpha$ a sequence of elements $\gamma_i(\alpha)\in B_\Gamma(r)$ such that $d_M(\gamma_i p_i,p_{i+1})<\frac{r}{t}$ for every $i$.
\label{constr:jumps}
\end{constr}

\begin{dfn} Let $t>\frac{2r}{\delta_\Gamma(2r)}$. Write $P(M;r)$ for the space of $r$-paths. Also write $F(B_\Gamma(r))$ for the free group on $B_\Gamma(r)$. Define
	$$\widetilde{Q}:P(M; r) \to F(B_\Gamma(r))$$
by $\widetilde{Q}(\alpha):=\gamma_\ell(\alpha)\cdots \gamma_0(\alpha)$ for $\alpha=(p_0,p_1,\dots,p_\ell)\in P(M;r)$.
\label{dfn:q}
\end{dfn}

\begin{lem} Let $t>\frac{4r}{\delta(4r)}$. If $\alpha$ and $\beta$ are $r$-homotopic $r$-loops based at $p_0$, then as elements of $\Gamma$ we have $\widetilde{Q}(\alpha)=\widetilde{Q}(\beta)$. Hence $\widetilde{Q}$ descends to a map
	$$Q:\pi_1(M,p_0 ; r) \to \Gamma.$$
For $\alpha\in\pi_1(M,p_0;r)$, we say $Q(\alpha)\in\Gamma$ is the \emph{orbital jump} along $\alpha$.
\label{lem:q_descent}
\end{lem}
\begin{proof} Clearly $\widetilde{Q}$ is invariant under endpoint appends and deletions. Therefore it suffices to show the statement for a pair of $r$-loops $\alpha, \beta$ based at $p_0$ that are $r$-close and of the same length. By Lemma \ref{lem:baby_step_htpy}, there is an $r$-homotopy from $\alpha$ to $\beta$ of the form
	$$\alpha=\eta_0, \, \eta_1, \, \dots, \, \eta_k =\beta$$
where $\eta_i$ and $\eta_{i+1}$ differ in at most two positions, and these are consecutive. Therefore it suffices to prove the statement assuming $\beta=\eta_1$.

So suppose
	$$\alpha=(p_0,\dots,p_{k-1}, p_k, p_{k+1}, p_{k+2}, \dots, p_\ell)$$
and
	$$\beta=(p_0,\dots,p_{k-1}, q_k, q_{k+1}, p_{k+2}, \dots, p_\ell).$$
Let $\gamma_{k-1}, \gamma_k, \gamma_{k+1}$ be defined as before, i.e. these are the unique elements of length at most $r$ such that
	$$d_M(\gamma_i p_i, p_{i+1})<\frac{r}{t}.$$
Likewise define $\eta_i$ associated to the path $\beta$. Since $\alpha$ and $\beta$ coincide except for the $k$ and $(k+1)$st position, we clearly have $\gamma_i=\eta_i$ for $|i-k|>1$. Therefore to prove the claim it suffices to show
	\begin{equation}
		\gamma_{k+1} \gamma_k \gamma_{k-1} = \eta_{k+1} \eta_k \eta_{k-1}.
	\label{eq:localprod}
	\end{equation}
Let $\xi_k$ and $\xi_{k+1}$ be the unique elements of length at most $r$ such that
	$$d_M(\xi_i q_i, p_i)<\frac{r}{t}.$$
	\begin{figure}[h]
		\includegraphics[height=5cm]{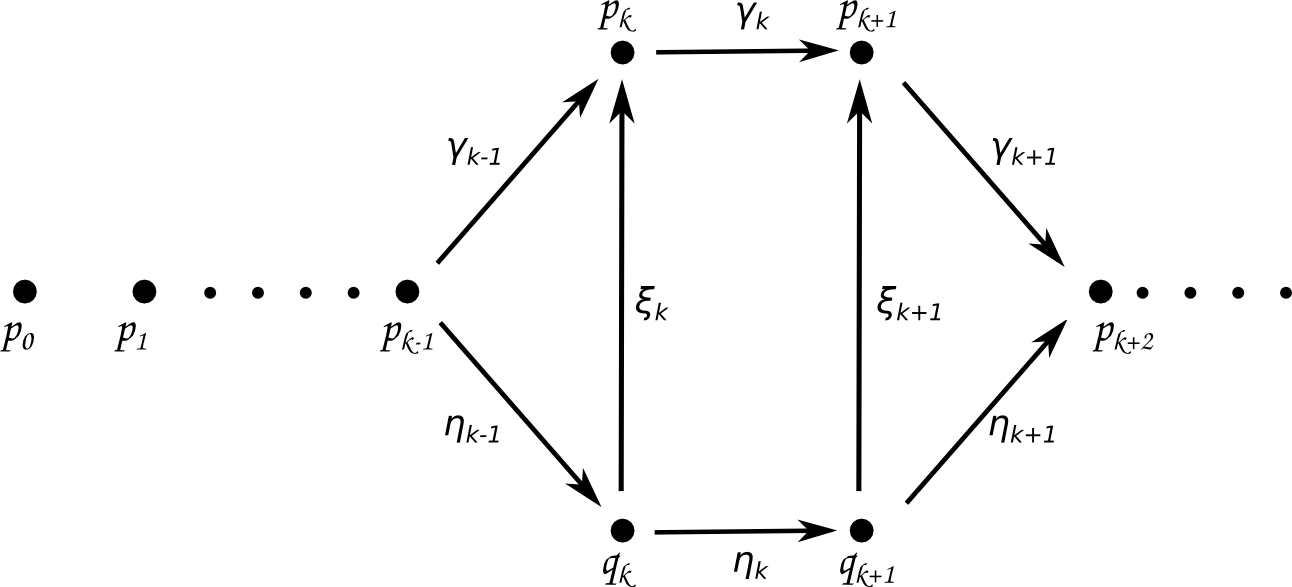}
		\caption{Proof of Lemma \ref{lem:q_descent}. All arrows represent points that are distance at most $r$ apart with respect to $d_t^\Gamma$.}
		\label{fig:localprod}
	\end{figure}
See Figure \ref{fig:localprod} for a schematic depiction of the situation. First observe that
	$$\gamma_{k-1}\eta_{k-1}^{-1}=\xi_k \hspace{1 cm} \text{and} \hspace{1 cm} \gamma_{k+1}^{-1}\eta_{k+1}=\xi_{k+1}.$$
Indeed otherwise one would obtain a nontrivial element of $\Gamma$ of length at most 3r that displaces a point less than $\frac{3r}{t}$, but this is impossible since 			
	$$t>\frac{4r}{\delta_\Gamma(4r)}>\frac{3r}{\delta_\Gamma(3r)}.$$
Hence to prove Equation \eqref{eq:localprod}, it suffices to show $\gamma_k \xi_k = \xi_{k+1}\eta_k.$ Indeed, $(\xi_{k+1}\eta_k)\inv\gamma_k\xi_k$ is an element of length at most $4r$ that displaces $q_k$ by less than $\frac{4r}{t}$. Since $t>\frac{4r}{\delta_\Gamma(4r)}$, we must have $(\xi_{k+1}\eta_k)\inv\gamma_k \xi_k=e$, as desired. \end{proof}

We make a new definition that will be useful for showing $Q$ is surjective. For $\gamma\in\Gamma$, we define an $r$-loop based at $p_0$ that first performs orbital jumps whose product in $\Gamma$ is $\gamma$, and then follows a discretization of a continuous path in $M$ from $\gamma p_0$ to $p_0$.

\begin{dfn}  For each $\gamma\in\Gamma$, make a choice of writing $\gamma=s_{i_k} \cdots s_{i_1}$ as a product of generators, and a choice of path $c_\gamma$ from $\gamma p_0$ to $p_0$. Further choose a discretization $(q_0, q_1, \dots, q_\ell)$ on the scale $\frac{r}{t}$ of $c_\gamma$, and such that $q_0=\gamma p_0$ and $q_\ell=p_0$. Define $\widetilde{\jmath}_\Gamma:\Gamma\to \Omega(M,p_0 ; r)$ by $\widetilde{\jmath}(\gamma)(0):=p_0$ and
	$$\widetilde{\jmath}(\gamma)(j):=\left( \prod_{1\leq l\leq j} s_{i_l} \right)  p_0$$
if $1\leq l\leq k$, and $\widetilde{\jmath}(\gamma)(k+j)=q_j$ if $k<j\leq k+\ell$. Write $\jmath_\Gamma:\Gamma\to\pi_1(M, p_0 ; r)$ for the map that assigns the coarse homotopy class.
\label{dfn:jgamma}
 \end{dfn}

%
\begin{warning} $\jmath_\Gamma$ need not be a homomorphism.
\label{warning:jgamma_nonhom}
\end{warning}
\begin{lem} $Q\circ \jmath_\Gamma =\text{id}_\Gamma$. Hence $\jmath_\Gamma$ is injective and $Q$ is surjective.\end{lem}
\begin{proof}This is a straightforward verification.\end{proof}

\subsection{Canonical forms for coarse paths} \label{sec:canonical} We briefly interrupt the computation of the coarse fundamental group to discuss a useful tool for the rest of the computation. Namely, we can reorder the jumps in an $r$-path so that all the orbital jumps occur at the start, and all the spatial ones at the end.
\begin{dfn} Let $\alpha$ be an $r$-path. Define the \emph{orbital component} of $\alpha$
	$$\calO(\alpha):=(p_0,\gamma_0 p_0, \gamma_1 \gamma_0 p_0, \dots, Q(\alpha) p_0)$$
and the \emph{spatial component}
	$$\calS(\alpha):=(Q(\alpha)p_0, \gamma_\ell \cdots \gamma_1 p_1, \dots, p_\ell)$$
of $\alpha$. The \emph{canonical form} of $\alpha$ is $\calC(\alpha):=\calO(\alpha)\ast\calS(\alpha)$.
\label{dfn:canonical}
\end{dfn}
The relevance for the present computation of coarse fundamental groups is that this operation does not change the coarse homotopy class.
\begin{lem} Let $t\gg r\gg 1$. Then any $r$-path $\alpha$ is $r$-homotopic rel endpoints to its canonical form $\calC(\alpha)$. \label{lem:canonical_htpy}\end{lem}
\begin{proof}  This is straightforward. For the idea to to construct the homotopy from $\alpha$ to $\calC(\alpha)$, see Figure \ref{fig:canonicalpath}.
	\begin{figure}[h]
		\includegraphics[height=10cm]{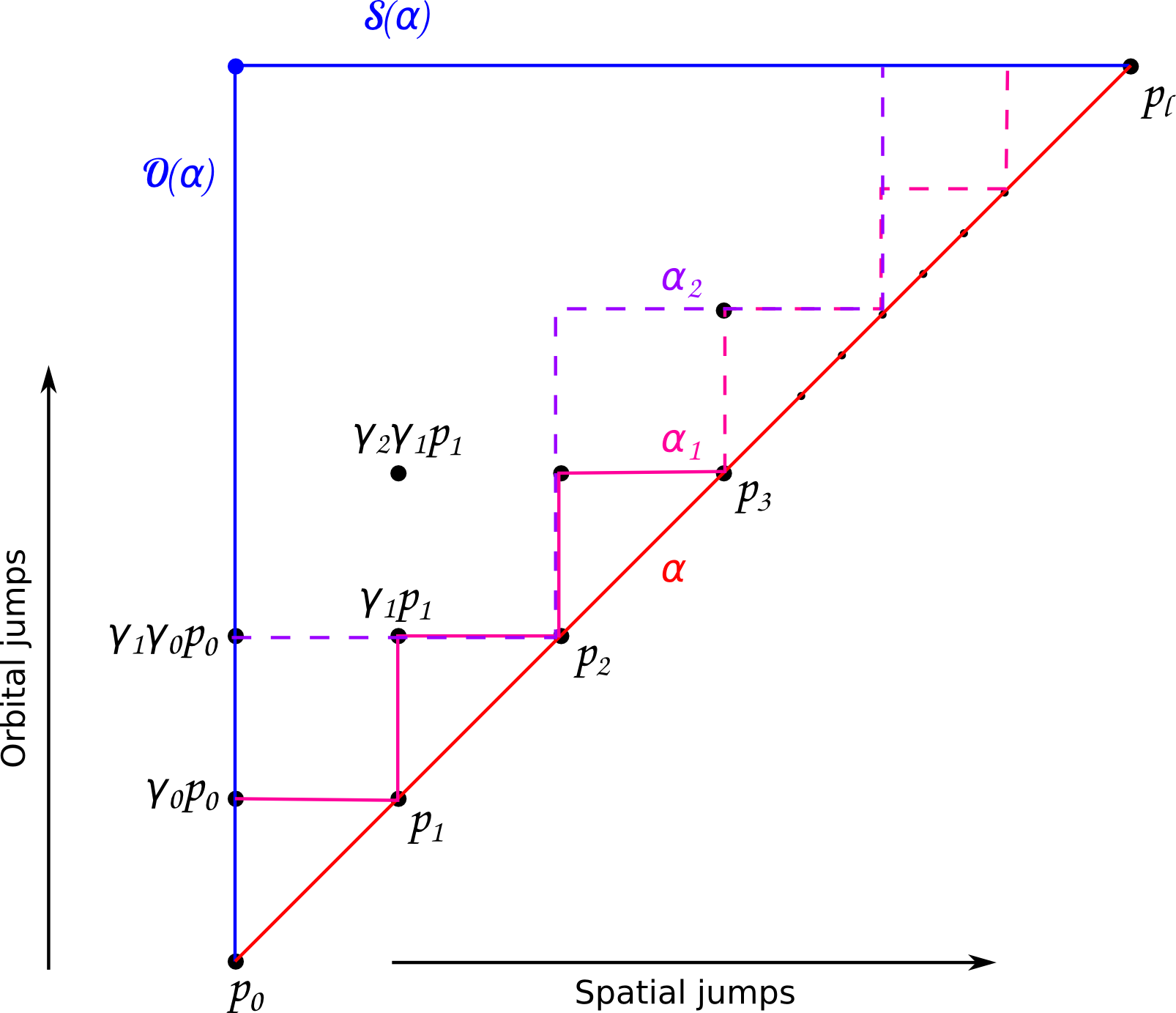}
		\caption{Homotopy of a loop to its canonical form. Horizontal jumps are distance at most $\frac{r}{t}$ with respect to $d_M$; Vertical jumps are distance at most $r$ with respect to word metric on $\Gamma$. Further $\alpha$ is the diagonal curve; $\calO(\alpha)$ is the left side; $\calS(\alpha)$ is the top side.}
		\label{fig:canonicalpath}		
	\end{figure} \end{proof}
For later use, we record the following elementary property of canonical forms:
\begin{lem} Let $t\gg r\gg 1$. Then an $r$-path and its canonical form have the same length.
\label{lem:canonical_length}
\end{lem}
\begin{proof} Let $\alpha$ be an $r$-path. We keep the notation of Definition \ref{dfn:canonical}. Note that for each $i$, we have
	$$d_t^\Gamma(p_i,p_{i+1})=\|\gamma_i\| + t\, d_M(\gamma_i p_i \,, \gamma_{i+1}),$$
so
	$$\ell(\alpha)=\sum_i (\|\gamma_i\| + t \, d_M(\gamma_i p_i \,,\gamma_{i+1})).$$
On the other hand, it is easy to see that
	$\ell(\calO(\alpha))=\sum_i \|\gamma_i\|$
and
	$\ell(\calS(\alpha))=\sum_i t\, d_M(\gamma_i p_i \,, p_{i+1}).$
\end{proof}
It is easy to verify the following cocycle-type behavior of canonical forms with respect to concatenation: 
\begin{lem} Let $\alpha$ and $\beta$ be $r$-paths such that the endpoint of $\alpha$ is distance at most $r$ from the starting point of $\beta$. Then
	$$\calO(\alpha\ast\beta)=\calO(\alpha)\ast\calO(\beta)$$
and
	$$\calS(\alpha\ast\beta)=(Q(\beta) \, \calS(\alpha))\ast \calS(\beta).$$
\label{lem:canonical_concat}
\end{lem}

\subsection{The contribution of the manifold} \label{sec:contrib_mnfd} We will now work to explain the relationship between $\pi_1(M)$ and $\pi_1(M ; r)$. Consider the natural discretization map
	$$\jmath_M:\pi_1(M)\hookrightarrow \pi_1(M_t ; r).$$
See Section \ref{sec:coarse_htpy} for the definition. We aim to show this is an embedding whose image is a normal subgroup.

\begin{lem} For $t\gg r\gg 1$, the map $\jmath_M$ is injective.
\label{lem:jm_no_ker}
\end{lem}
\begin{proof} Let $c\in \pi_1(M)$ be a nontrivial (continuous) homotopy class of loops. Let $\theta$ be a discretization on the scale $\frac{r}{t}$ of an element of $c$. Suppose $\theta$ were $r$--nullhomotopic. By Lemma \ref{lem:baby_step_htpy}, we can choose an $r$--nullhomotopy of $\theta$ such that any two consecutive curves in the homotopy differ in at most two positions, and these are consecutive. Let $\alpha$ and $\beta$ be any two consecutive curves in the nullhomotopy. We claim that $\calS(\alpha)$ and $\calS(\beta)$ are $\frac{4r}{t}$--close as coarse paths in $M$.

To see this, retain the notation of the proof of Lemma \ref{lem:canonical_length} (see Figure \ref{fig:localprod}). Note that since $Q(\theta)=e$, and hence also $Q(\alpha)=Q(\beta)=e$, we know that $\calS(\alpha)$ and $\calS(\beta)$ start at $p_0$. Now we have
	$$\calS(\alpha)=(\gamma_\ell \cdots \gamma_0 p_0, \gamma_\ell \cdots \gamma_1 p_1, \gamma_\ell \cdots \gamma_2 p_2, \dots, \gamma_\ell p_\ell)$$
and similarly for $\calS(\beta)$. Now since $\eta_j = \gamma_j$ whenever $|j-k|>1$, we see that $\calS(\alpha)$ and $\calS(\beta)$ agree for the last $\ell-(k+1)$ terms. Further we have $\gamma_{k+1} \gamma_k \gamma_{k-1} = \eta_{k+1}\eta_k \eta_{k-1}$ (see Equation \eqref{eq:localprod}), so $\calS(\alpha)$ and $\calS(\beta)$ also agree for the first $k$ terms. To finish the proof, we will show
	$$d_M(\gamma_\ell \cdots \gamma_{k+2} \gamma_{k+1} p_{k+1} \, , \, \gamma_\ell \cdots \gamma_{k+2} \eta_{k+1} q_{k+1})<\frac{2r}{t}$$
and
	$$d_M(\gamma_\ell \cdots \gamma_{k+2} \gamma_{k+1} \gamma_k p_{k} \, , \, \gamma_\ell \cdots \gamma_{k+2} \eta_{k+1} \eta_k q_{k})<\frac{4r}{t}.$$
The first inequality is equivalent to proving
	$$d_M(\gamma_{k+1} p_{k+1} \, , \, \eta_{k+1} q_{k+1})<\frac{2r}{t},$$
which is true because both are distance at most $\frac{r}{t}$ from $p_{k+2}$. The second is equivalent to proving
	$$d_M(\gamma_{k+1} \gamma_k p_k \, , \, \eta_{k+1} \eta_k q_k)<\frac{4r}{t},$$
which is true because both are distance at most $\frac{2r}{t}$ from $p_{k+1}$.

By connecting the points in $\calS(\alpha)$ (resp. $\calS(\beta)$) by minimizing geodesic segments, we obtain a continuous loop $c_\alpha$ (resp. $c_\beta$) in $M$. Since $c_\alpha$ and $c_\beta$ are uniformly close, we see that for $t\gg r\gg 1$, they are homotopic rel $p_0$.

It follows that for any coarse loop $\alpha$ in the nullhomotopy of $\theta$, we have that $\calS(\alpha)$ is in the image under $\jmath_M$ of the original loop $[c]$. In particular, the length of $\calS(\alpha)$ is uniformly bounded below. This contradicts the fact that we have a coarse nullhomotopy of $\theta$.\end{proof}

%
%
%

\subsection{Relationship between the two contributions} \label{sec:mnfd_gp_rel} Let us now return to the computation of the coarse fundamental group.
\begin{lem} $\jmath_M(\pi_1(M))$ and $\jmath_\Gamma(\Gamma)$ generate $\pi_1(M_t,p_0;r)$ and have trivial intersection.
\label{lem:coarse_gen}
\end{lem}
\begin{proof} Let $\alpha$ be an $r$--loop based at $p_0$. By Lemma \ref{lem:canonical_htpy}, $\alpha$ is $r$-homotopic rel endpoints to its canonical form $\calC(\alpha)=\calO(\alpha)\ast \, \calS(\alpha)$. Consider $\widetilde{\jmath}_\Gamma(Q(\alpha)^{-1})\ast \calC(\alpha)$. Note that $\widetilde{\jmath}_\Gamma(Q(\alpha)^{-1})$ ends in a sequence of orbital jumps $\calO(\widetilde{\jmath}_\Gamma(Q(\alpha)^{-1}))$ whose product in $\Gamma$ is $Q(\alpha)^{-1}$ and hence
	$$\calO(\widetilde{\jmath}_\Gamma(Q(\alpha)))\ast \calO(\alpha)$$
is a sequence of orbital jumps by products of conjugates of relators. Since $r$ is larger than the length of any relator, we can cancel these out and hence $\widetilde{\jmath}_\Gamma(Q(\alpha)^{-1})\ast \calC(\alpha)$ is $r$-homotopic to a spatial loop, i.e. is in the image of $\jmath_M$.

Finally, we show $\jmath_M(\pi_1(M))$ and $\jmath_\Gamma(\Gamma)$ intersect trivially. Indeed, we clearly have $Q\circ \jmath_M=1$. On the other hand, if $\gamma\in\Gamma$ is nontrivial, then $Q(\jmath_\Gamma(\gamma))=\gamma$, so $\jmath_\Gamma(\gamma)\notin \jmath_M(\pi_1(M))$.
\end{proof}

\begin{lem} \label{lem:normal} $\jmath_M(\pi_1(M))$ is normal in $\pi_1(M, p_0; r)$.\end{lem}

\begin{proof} We claim that $\jmath_M(\pi_1(M))=\ker(Q)$, which will suffice to show $\jmath_M(\pi_1(M))$ is normal. Indeed, it is clear that $\jmath_M(\pi_1(M))\subseteq \ker(Q)$. To show the other direction, suppose $\alpha$ is an $r$-loop based at $p_0$ with $[\alpha]\in\ker(Q)$. Let $\widetilde{\alpha}$ be the canonical representative, where the initial sequence of jumps is
	$$\gamma_0, \gamma_1,\dots,\gamma_\ell$$
and $\gamma_i \in S$, the generators for $\Gamma$, and after that $\widetilde{\alpha}(i)=p_i$, so $d_M(p_i,p_{i+1})<\frac{r}{t}$. We know that $Q(\alpha)=\prod_i \gamma_i$ is trivial, so as a word in $F(S)$, we have that $\prod_i \gamma_i$ is a product of conjugates of relators. Since $r$ is larger than the maximal size of a relator, and $r>2$ (so that for any generator $s$, a jump sequence of the form $ss\inv$ can be canceled), we see that $\widetilde{\alpha}$ is $r$-homotopic to the loop that omits the initial sequence of jumps and then performs the rest of $\widetilde{\alpha}$, i.e. $\calS(\widetilde{\alpha})$. This loop is in $\jmath_M(\pi_1(M))$, as desired. \end{proof}

\subsection{End of the proof} \label{sec:ext_lift} At this point we know that for $t\gg r \gg 1$, the group $\pi_1(M_t , p_0 ; r)$ is an extension
	\begin{equation}
		1\to \pi_1(M,p_0) \overset{\jmath_M}{\longrightarrow} \pi_1(M_t , p_0 ; r) \overset{Q}{\longrightarrow} \Gamma \to 1.
		\label{coarse_ses}
	\end{equation}
Further the action of $\Gamma$ by conjugation on $\pi_1(M,p_0)$ is induced from $\Gamma\curvearrowright M$. We want to show $\pi_1(M_t; r)$ is given by the group of lifts $\widetilde{\Gamma}$ of $\Gamma$ to the universal cover $\widetilde{M}$ of $M$. Note that $\widetilde{\Gamma}$ is also an extension
	\begin{equation}
	1\to \pi_1(M)\overset{\iota}{\longrightarrow} \widetilde{\Gamma} \overset{\pi}{\longrightarrow} \Gamma\to 1,
	\label{lift_ses}
	\end{equation}
with the action of $\Gamma$ on $\pi_1(M)$ by conjugation induced from $\Gamma$ acting on $M$. We will need the following straightforward fact:

\begin{lem} The action of $\widetilde{\Gamma}$ on $\widetilde{M}$ is free.\end{lem}
\begin{proof} This is immediate from the fact that $\Gamma$ acts freely on $M$, and $\pi_1(M)$ acts freely on $\widetilde{M}$. Indeed, suppose $\widetilde{\gamma}\in\widetilde{\Gamma}$ fixes $\widetilde{p}\in\widetilde{M}$. Let $\gamma=\pi(\widetilde{\gamma})$ be the image of $\widetilde{\gamma}$ in $\Gamma$ and let $p$ be the image of $\widetilde{p}$ in $M$. Then $\gamma$ fixes $p$. Since $\Gamma$ acts freely, it follows $\gamma=e$, and hence $\widetilde{\gamma}\in\pi_1(M)$. But $\pi_1(M)$ acts freely on $\widetilde{M}$, so $\widetilde{\gamma}=e$.\end{proof}

\begin{constr} Fix a basepoint $\widetilde{p}_0\in\widetilde{M}$ that maps to $p_0\in M$. We define a map
	$$\Theta: \pi_1(M, p_0 ; r)\to \widetilde{\Gamma}$$
as follows. Let $\alpha\in \pi_1(M, p_0 ; r)$. We assume $\alpha=\calO(\alpha)\ast\calS(\alpha)$ is in canonical form and write $\gamma:=Q(\alpha)$. Write $c$ for the piecewise geodesic path in $M$ from $p_0$ to $\gamma p_0$ obtained by connecting the points in the spatial path $\calS(\alpha)$ by minimizing geodesic segments (observe that since $\frac{r}{t}<\text{injrad}(M)$, we can do this in a unique way). Now lift $c$ to a path $\widetilde{c}$ in the universal cover $\widetilde{M}$ that starts at $\widetilde{p}_0$. Note that $\widetilde{c}(1)$ maps to $c(1)=\gamma p_0\in M$. Hence there exists an element $\widetilde{\gamma}\in\widetilde{\Gamma}$ with $\widetilde{\gamma} \widetilde{p}_0=\widetilde{c}(1)$. Since $\widetilde{\Gamma}$ acts freely on $\widetilde{M}$ (by the previous claim), $\widetilde{\gamma}$ is unique. Define
	$$\Theta(\alpha):=\widetilde{\gamma}.$$
\label{constr:theta}
\end{constr}
Recall that $\widetilde{\Gamma}$ is an extension
\begin{equation}
	1\to \pi_1(M)\overset{\iota}{\longrightarrow} \widetilde{\Gamma} \overset{\pi}{\longrightarrow} \Gamma\to 1.
	\end{equation}
\begin{lem} $\pi\circ\Theta=Q$ and $\Theta\circ \jmath_M=\iota$.
\label{lem:theta_lifts}
\end{lem}
\begin{proof} Retain the notation of Construction \ref{constr:theta}. For the first claim: Since $\Gamma$ acts freely, it suffices to show that the image of $\Theta(\alpha)\widetilde{p}_0$ in $M$ is exactly $\gamma p_0$. But $\Theta(\alpha)\widetilde{p}_0=\widetilde{c}(1)$ projects to $c(1)=\gamma p_0$ in $M$, as desired.

The second claim is clear because a lift of $c\in \pi_1(M,p_0)$ starting at $\widetilde{p}_0$ has endpoint given by $\iota(c)\widetilde{p}_0$.\end{proof}

\begin{lem} $\Theta$ is an antihomomorphism.\end{lem}
\begin{proof} Let $\alpha, \beta$ be $r$--loops based at $p_0$ and assume they are in canonical form. Recall from Lemma \ref{lem:canonical_concat} that $\calO(\alpha\ast\beta)=\calO(\alpha)\ast\calO(\beta)$ and $\calS(\alpha\ast\beta)=(Q(\beta) \, \calS(\alpha))\ast \calS(\beta)$. Let $c_\alpha$ (resp. $c_\beta$) be the (continuous) path in $M$ obtained by connecting the points of $\calS(\alpha)$ (resp. $\calS(\beta)$) by minimizing geodesic segments.

The lift of the path $\gamma_\beta \, c_\alpha \ast c_\beta$ to $\widetilde{M}$ that starts at $\widetilde{p}_0$ starts with $\widetilde{c}_\beta$, and then takes the lift of $\gamma_\beta \, c_\alpha$ that starts at $\widetilde{c}_\beta(1):=\Theta(\beta) \, \widetilde{p}_0$. Therefore the lift of $\gamma_\beta \, c_\alpha$ that is taken is $\Theta(\beta) \, \widetilde{c}_\alpha$. The endpoint is therefore $\Theta(\beta) \, \widetilde{c}_\alpha(1) = \Theta(\beta) \, \Theta(\alpha) \widetilde{p}_0$. Hence we see $\Theta(\alpha\beta)=\Theta(\beta) \, \Theta(\alpha)$.
%
%
\end{proof}

Finally we are able to prove that $\pi_1(M_t;r)\cong\widetilde{\Gamma}$, thereby finishing the proof of Proposition \ref{prop:coarsepi1}.

\begin{lem} $\Theta$ is an anti-isomorphism.
\label{lem:phi_isom}
\end{lem}
\begin{proof} We first show $\Theta$ is injective. Let $\alpha\in \ker(\Theta)$. Since $\Theta(\alpha)$ lifts $Q(\alpha)$ and $\Theta(\alpha)=e$, we see $\gamma=e$, so $\alpha\in \jmath_M(\pi_1(M))$. But by Claim \ref{lem:theta_lifts}, we have $\Theta\circ \jmath_M=\iota$, and $\iota$ is injective. It follows that $\ker\Theta\cap \jmath_M(\pi_1(M))=1$. This proves $\Theta$ is injective.

Finally we show $\Theta$ is surjective. Let $\widetilde{\gamma}\in\widetilde{\Gamma}$, and write $\gamma=\pi(\widetilde{\gamma})$ for the image of $\widetilde{\gamma}$ in $\Gamma$. Note that $\Theta(\jmath_\Gamma(\gamma))$ is characterized by the property that $\Theta(\jmath_\Gamma(\gamma))\widetilde{p}_0=\widetilde{c}_\gamma(1)$, where $c_\gamma$ is the arbitrary choice of continuous path from $p_0$ to $\gamma p_0$ made in the definition of $\jmath$, see Definition \ref{dfn:jgamma}. Since $\widetilde{c}_\gamma(1)$ and $\widetilde{\gamma}\widetilde{p}_0$ both project to $\gamma p_0\in M$, there is $\eta\in\pi_1(M,p_0)$ such that $\iota(\eta)\widetilde{c}_\gamma(1)=\widetilde{\gamma}\widetilde{p}_0$. Hence we have
	$$\eta \, \Theta(\jmath_\Gamma(\gamma))=\widetilde{\gamma}.$$
On the other hand $\iota(\eta)=\Theta(\jmath_M(\eta))$ by Lemma \ref{lem:theta_lifts}. Hence we see
	$$\Theta(\jmath_M(\eta)) \, \Theta(\jmath_\Gamma(\gamma))=\widetilde{\gamma},$$
so $\Theta$ is surjective.\end{proof}

\subsection{Coarse homotopical behavior of quasi-isometries} The stabilization of $\pi_1(M_t; r)$ to $\widetilde{\Gamma}$ as $r$ is fixed and $t\to\infty$ immediately implies that the change-of-scale map becomes an isomorphism. More precisely, let us decorate the isomorphism $\Theta$ defined above with the scale $r$, so that $\Theta_r : \pi_1(M_t ; r)\to \widetilde{\Gamma}$ is the above isomorphism, and let $\mu_{r,r'}$ be change-of-scale from $r$ to $r'$. Then whenever $t\gg r'>r\gg 1$, we have $\Theta_{r'}\circ \mu_{r,r'}=\Theta_r.$ In particular $\mu_{r,r'}$ is an isomorphism. This has the following interesting consequence for the behavior of quasi-isometries of level sets of warped cones.
\begin{lem} Fix $L\geq 1$ and $C\geq 0$. For $s,t\gg r\gg 1$, any $(L,C)$--quasi-isometry $f:M_t\to N_s$ induces an isomorphism $\Phi:\pi_1(M_t; r)\to \pi_1(N_s; Lr+C)$ on coarse fundamental groups.
\label{lem:coarse_isom}
\end{lem}
\begin{proof} Write $g:N_s\to M_t$ for the coarse inverse of $f$, so $g$ is also an $(L,C)$--quasi-isometry and $f\circ g$ and $g\circ f$ displace points by at most $C$. Then
	$$g_\ast f_\ast =(g\circ f)_\ast : \pi_1(M_t ; r)\to \pi_1(M_t; L^2 r + LC + C)$$
is given by change-of-scale (because $g\circ f$ displaces points at most $C$ and the scale in the target is at least $C$), and hence $g_\ast f_\ast$ is an isomorphism. Likewise $f_\ast g_\ast$ is an isomorphism, so $f_\ast$ is an isomorphism. \end{proof}

\section{Coarse homotopy invariance of group contributions}
\label{sec:no_obstruct}

\subsection{Setup} For the rest of this section, let $\Gamma$ (resp. $\Lambda$) be a finitely presented group acting minimally, isometrically, and freely on a closed manifold $M$ (resp. $N$), and assume that $\Gamma$ and $\Lambda$ are not commensurable to a group with a free abelian factor. Further $f : M_t \to N_s$ will denote an $(L,C)$--quasi-isometry.

\subsection{Statement of result} In this section, we aim to define, for every $s,t\gg r\gg L+C$, an invariant of an $(L,C)$-quasi-isometry $f:M_t\to N_s$ that measures the extent to which $f$ (coarsely) preserves the action and manifold contributions to the coarse fundamental group, i.e. the data of the extensions
	$$1\to \pi_1(M)\to \pi_1(M_t; r)\to \Gamma\to 1$$
and the corresponding extension for $N_s$. To start, recall that we defined for an $r$-loop $\alpha$ in $M$, its orbital jump $\widetilde{Q}(\alpha)$ (see Definition \ref{dfn:q}). We intend to show:
	\begin{prop} There are commensurations $M'$ of $M$ and $N'$ of $N$, such that the following holds for any $s,t\gg r\gg 1$: Let $\Gamma'$ (resp. $\Lambda'$) be the group of lifts of $\Gamma$ (resp. $\Lambda$) to $M'$ (resp. $N'$), and let $f':M_t'\to N_s'$ be an $(L,C)$--quasi-isometry. Let $p,q\in M'$ and let $\alpha$ be a (continuous) path in $M'$ from $p$ to $q$. Then $Q_{\Lambda'}(f'\circ \jmath_M(\alpha))$ only depends on $p$ and $q$, and not on the path $\alpha$ chosen.
	\label{prop:myjumpisyourjump}
	\end{prop}
This will use in an essential way that the groups $\Gamma$ and $\Lambda$ are complicated (namely not commensurable to any group with a free abelian factor), and is false otherwise.

\subsection{Construction of the commensurable spaces} \label{sec:constr_comm} We start by finding the right covers $M'$ and $N'$ to prove Proposition \ref{prop:myjumpisyourjump}.

\begin{prop} $\Gamma$ contains a maximal finite normal subgroup. The same holds for $\Lambda$.
\label{prop:max_normal_fin}
\end{prop}
It clearly suffices to prove the statement for $\Gamma$. We prove this in two steps. We will start by proving this for a finite-index subgroup of $\Gamma$. Namely, let $G=\overline{\Gamma}$ denote the closure of $\Gamma$ in $\Isom(M)$, let $G^0$ be the connected component of the identity of $G$, and set $\Gamma_0:=\Gamma\cap G^0$. Note that $\Gamma_0$ is a finite-index subgroup of $\Gamma$.
\begin{lem} $\Gamma_0$ contains a maximal finite normal subgroup.
\label{lem:max_normal_fin0}
\end{lem}
\begin{proof} Note that any finite normal subgroup $F$ of $\Gamma_0$ is central, because $G^0$ normalizes $F$ and by connectedness must centralize it. Since a pair of finite central subgroups generates a finite central subgroup, we see the union $F$ of all finite normal subgroups is a central locally finite subgroup of $\Gamma$. We want to show it is finite.

Indeed, $\overline{F}^0$ is a compact connected normal abelian subgroup of the compact connected Lie group $G^0$, potentially trivial, and contains a dense locally finite group $F_0:=S\cap \overline{F}^0$. Then $\overline{F}^0$ is a local torus factor of $G$. Consider the projection onto this factor, i.e.
	$$G^0\to \overline{F}^0\slash A$$
where $A$ is a finite group. The image of $\Gamma$ is then a finitely generated abelian group in the torus $\overline{F}^0\slash A$ and contains the image of $F_0$, which is isomorphic to $F_0\slash (F_0\cap A)$. But any locally finite subgroup of a finitely generated abelian group is finite, so $F_0\slash(F_0\cap A)$ is finite. Since $A$ is also finite, it follows that $F_0$ is finite. Finally, $\overline{F}$ has finitely many components given by the group $F\slash F_0$, so that $F$ is an extension of the finite group $F\slash F_0$ by the finite group $F_0$, and hence $F$ is finite, as desired.\end{proof}
We can now show $\Gamma$ contains a maximal finite normal subgroup.
\begin{proof}[Proof of Proposition \ref{prop:max_normal_fin}] Let $F_0$ be the maximal finite normal subgroup of $\Gamma_0$, which exists by the above Lemma \ref{lem:max_normal_fin0}. Since any pair of finite normal subgroups generates a finite normal subgroup, the union $F$ of all finite normal subgroups of $\Gamma$ is well-defined.

We have to show $F$ is finite. Clearly $F_0\subseteq F$ and $F\cap G^0$ is a finite normal subgroup of $\Gamma_0$, so we must have $F\cap G^0 =F_0$. Therefore projection mod $G^0$ gives an embedding
	$$F\slash F_0\hookrightarrow G\slash G^0,$$
and hence
	$$|F|\leq |F_0|\cdot |G\slash G^0|<\infty,$$
as desired.\end{proof}
Armed with Proposition \ref{prop:max_normal_fin}, we can construct the right homogeneous spaces for the proof of Proposition \ref{prop:myjumpisyourjump}.
\begin{constr}
Let $F_\Gamma$ be the maximal finite normal subgroup of $\Gamma$. Then $\Gamma /F_\Gamma$ has no finite normal subgroup. Further since $F_\Gamma$ is normal, the action of $\Gamma$ descends to a minimal, isometric, free action of $\Gamma':=\Gamma\slash F_\Gamma$ on $M':=M/F_\Gamma$. Likewise we set $\Lambda':=\Lambda\slash F_\Lambda$ and $N':=N\slash F_\Lambda$. By Lemma \ref{lem:cone_fin_quot}, the warped cone over a finite quotient is uniformly quasi-isometric to the original warped cone. 
\label{constr:quot_by_max_fin}
\end{constr}

\subsection{Proof of Proposition \ref{prop:myjumpisyourjump}} We now replace our initial actions by the commensurable quotients just constructed and so denote these quotients $\Gamma', M', \Lambda', N'$ by $\Gamma, M, \Lambda, N$. The proof proceeds in several steps. We start by reducing Proposition \ref{prop:myjumpisyourjump} to a statement on coarse fundamental groups (i.e. about loops, not paths).
	\begin{lem} To prove Proposition \ref{prop:myjumpisyourjump}, it suffices to show that for any (continuous) loop $\alpha$ in $M$, we have $Q_\Lambda(f\circ \widetilde{\jmath}_M(\alpha))=e$.
	\label{cl:reduct}
	\end{lem}
	\begin{proof} Let $p,q\in M$ and let $\alpha, \beta$ be two paths from $p$ to $q$. Recall that $\overline{\beta}$ denotes the reverse of $\beta$. It is immediate from the definition that $Q_\Lambda(f\circ\overline{\beta})=Q_\Lambda(f\circ\beta)^{-1}$ and
		$$Q_\Lambda(f\circ\beta)^{-1} \, Q_\Lambda(f\circ\alpha)	=	Q_\Lambda(f\circ\overline{\beta}) \, Q_\Lambda(f\circ\alpha)
													=	Q_\Lambda(f\circ(\overline{\beta}\ast\alpha)).$$
	Since $\overline{\beta}\ast\alpha$ is a loop in $M$, we have $Q_\Lambda(f\circ(\overline{\beta}\ast\alpha))=e$ by assumption and hence $Q_\Lambda(f\circ\alpha)=Q_\Lambda(f\circ\beta)$.
	\end{proof}
Let $\Phi:=f_{\ast}$ be the map on coarse fundamental groups induced by the quasi-isometry $f:M_t\to N_s$. Consider the composition
	$$a:\pi_1(M,p_0)\overset{\jmath_M}{\hookrightarrow} \pi_1(M_{t}, p_0; r)\overset{\Phi}{\longrightarrow} \pi_1(N_{s}, f(p_0); Lr+C)\overset{Q_\Lambda}{\twoheadrightarrow} \Lambda.$$
To prove Proposition \ref{prop:myjumpisyourjump}, it therefore suffices to show:
\begin{lem} $a$ is trivial.
\label{lem:no_obstr}
\end{lem}
Before we prove this, let us just note:
\begin{rmk} This is the key place in the proof of Theorem \ref{thm:main} where we use that $\Gamma$ and $\Lambda$ have no free abelian factor (up to commensuration). In a sense to be made precise below, the isomorphism on coarse fundamental group induced by $f_t$ preserves the manifold contribution $\jmath_M(\pi_1(M))$ and the group contribution $\jmath_\Gamma(\Gamma)$ to the coarse fundamental group.

This fails disastrously for actions of free abelian groups, and gives rise to the examples of Kim and Kielak--Sawicki of quasi-isometric warped cones for actions of groups that are not commensurable.
\label{rmk:ess_hyp}
\end{rmk}

%
%
%

\begin{proof}[Proof of Lemma \ref{lem:no_obstr}]

Since $\jmath_M(\pi_1(M))$ is normal in $\pi_1(M_{t};r)$, it is clear that the image of $a$ is normal in $\Lambda$. Since $\Lambda$ does not contain any nontrivial finite normal subgroup (see Construction \ref{constr:quot_by_max_fin}), it therefore suffices to show that the image of $a$ is also finite.

Since $M$ is a compact homogeneous space, $\pi_1(M)$ contains a finite-index central free abelian subgroup $A$. We claim that there is a finite-index subgroup $\Delta$ of $\pi_1(M_{t};r)$ such that $A$ is also central in $\Delta$. To see this, recall that $\pi_1(M_{t};r)$ is isomorphic to the group of lifts of $\Gamma$ to $M$ and hence is given by an extension
	$$1\to \pi_1(M)\to \pi_1(M_{t};r)\to \Gamma\to 1.$$
Since $\Gamma$ acts isometrically on $M$, the conjugation action $\pi_1(M_{t};r)\to\Gamma\to \Out(\pi_1(M))$ has finite image. Let $\Delta$ be the kernel, so $\Delta$ is of finite index in $\pi_1(M_{t};r)$ and centralizes $A$. Let $\Lambda':=Q_\Lambda(\Phi(\Delta))$ be the image of $\Delta$ in $\Lambda$ and write $B:=a(A)$. So $\Lambda'$ is a finite-index subgroup of $\Lambda$ that centralizes $B$.

In order to show $a$ has finite image, it suffices to show $B$ is finite (since $A$ is of finite index in $\pi_1(M)$). To see this, we claim that a finite-index subgroup of $B$ is a (necessarily abelian) factor of a finite-index subgroup of $\Lambda'$. Since no finite-index subgroup of $\Lambda$ has an abelian factor, it follows that $B$ must have been finite. To produce this factor of $\Lambda'$, it suffices to find a projection of a finite-index subgroup of $\Lambda'$ onto a finite-index subgroup of $B$.

To find this projection, consider the closure $\overline{B}$ of $B$ in Isom$(N)$. If $B$ is infinite, then $\overline{B}$ has nontrivial connected component $\overline{B}^0$. Further $\overline{B}^0$ is central in the compact Lie group $\overline{\Lambda'}^0$. Hence $\overline{B}^0$ is a local torus factor of $\overline{\Lambda'}^0$. Write $\pi :\overline{\Lambda'}^0\to\overline{B}^0\slash F$ (where $F$ is a finite group) for the projection to this torus factor, and write $T:=\overline{B}^0\slash F$ for the quotient.

Then $\pi(\Lambda')$ is a finitely generated abelian subgroup of $T$ containing the image $B\slash(B\cap F)$ of $B$, which is also finitely generated and abelian. There is a finite-index subgroup $B'$ of $B$ that intersects $F$ trivially and hence projects isomorphically to $B\slash(B\cap F)$. Therefore the following fact finishes the proof:

{\bf Fact.} Let $A_1$ be a finitely generated abelian group with a subgroup $A_2$. Then some finite-index subgroup of $A_2$ is a factor of some finite-index subgroup of $A_1$.
\end{proof}

Vanishing of $a$ has the following immediate consequence:
\begin{constr} The map $\Phi:\widetilde{\Gamma}\to\widetilde{\Lambda}$ induced by $f$ on coarse fundamental groups restricts to an isomorphism $\Phi|_{\pi_1(M)}:\pi_1(M)\to\pi_1(N)$ and hence descends to an isomorphism $\varphi:\Gamma\to\Lambda$.
\label{constr:group_map}
 \end{constr}

\section{Stabilization of maps on coarse fundamental groups}
\label{sec:stab}

At this point we are able to state the full version of Theorem \ref{thm:main_FSL}, which requires some technical assumptions.

\begin{thm} Let $\Gamma$ and $\Lambda$ be finitely presented groups acting isometrically, freely and minimally on compact manifolds $M$ and $N$ (respectively). 	Assume
	\begin{enumerate}[(i)]
		\item neither of $\Gamma$ and $\Lambda$ is commensurable to a group with a nontrivial free abelian factor,
		\item there exist uniform quasi-isometries $f_n: M_{t_n}\to N_{s_n}$ for two sequences $t_n, s_n\to\infty$ , and
		\item letting $\varphi_n:\Gamma\to\Lambda$ be the isomorphisms defined in Construction \ref{constr:group_map}, assume these stabilize up to conjugacy along a subsequence. \label{item:stab_hyp}
	\end{enumerate}
Then there is an affine commensuration of the action of $\Gamma$ on $M$ and the action of $\Lambda$ on $N$.
	\label{thm:main_tech}
	\end{thm}
In general, the sequences $\{t_n\}_n$ and $\{s_n\}_n$ could be extremely sparse, and hence without the coarse homotopical stabilization hypothesis \eqref{item:stab_hyp}, there is no relationship between the maps $f_n$. Of course, if the maps $f_n$ are all induced by a single quasi-isometry $f:\GGM\to \GLN$ of the geometric cones, then there is an extremely close connection between the different $f_n$ (although they are only uniformly coarsely Lipschitz, see Lemma \ref{lem:global_to_local}), and we will show this indeed implies that the maps $\varphi_n$ stabilize in this case (see Section \ref{sec:geomcone_stab} below). After establishing this stabilization phenomenon, the proofs of Theorem \ref{thm:main}, which shows affine commensurability of actions with quasi-isometric geometric cones, and Theorem \ref{thm:main_tech}, which shows affine commensurability only assuming quasi-isometric sequences of level sets but under the additional coarse stabilization hypothesis \ref{thm:main_tech}\eqref{item:stab_hyp}, will proceed completely parallel to each other (see Sections \ref{sec:semiconj} and \ref{sec:conj}).

In the next subsection, we will show that if $\Gamma$ has Property $\FSL$, then a sequence of uniform quasi-isometries satisfies the coarse stabilization hypothesis \ref{thm:main_tech}\eqref{item:stab_hyp}. Therefore Theorem \ref{thm:main_FSL}, obtaining affine commensurability for assuming quasi-isometric sequences of level sets under the additional assumption of Property $\FSL$, is a direct consequence of Theorem \ref{thm:main_tech} and the results of the next subsection.

Theorem \ref{thm:main}, which assumes quasi-isometric geometric cones, does not follow from Theorem \ref{thm:main_tech}, because quasi-isometries of geometric cones do not yield uniform quasi-isometries of level sets, but only uniformly coarsely Lipschitz maps as described in Lemma \ref{lem:global_to_local}.  It turns out that condition suffices to construct the conjugacy in both Theorem \ref{thm:main} and \ref{thm:main_tech} and we carry this out in Section \ref{sec:semiconj} and \ref{sec:conj}.


%
%

On the other hand, for the study of the geometry of expanders, and in particular for the proof of Theorems \ref{thm:expanders} and \ref{thm:superexpanders}, we need to control the behavior of subsequences of level sets of warped cones, which could be very sparse. Hence for these results Theorem \ref{thm:main} does not suffice, and one needs Theorem \ref{thm:main_FSL} and therefore Theorem \ref{thm:main_tech} and the results of the next few sections.


In this section, we will prove that in many cases, the stabilization hypothesis \eqref{item:stab_hyp} is satisfied.  In addition to the setting of the geometric cone discussed above, we will consider several easily checked group-theoretic assumptions on $\Gamma$ and $\Lambda$ that imply stabilization.  We introduce an abstract property of a finitely generated group which we call $\FSL$ which implies stabilization and provide many examples of groups satisfying $\FSL$.  We make no attempt
to be exhaustive and simply prove $\FSL$ in several cases related to applications. The following groups have $\FSL$:
	\begin{enumerate}[(1)]
		\item 	\label{item:finite_out} 	Any $\Gamma$ with $\Out(\Gamma)$ finite.
		\item		\label{item:surface}		Surface groups.
		\item 	\label{item:free}		Free groups.
		\item 	\label{item:free_prod}	Certain free products.
		\item		\label{item:direct_prod}	Certain direct products.
	\end{enumerate}

We will use the first or third, and fifth criterion in the proof of Theorem \ref{thm:expanders} and the first and fifth in the proof of Theorem \ref{thm:superexpanders}.  We believe one can give slightly different proofs of those theorems using the fourth condition in place of the fifith, but as the proof using the fifth is shorter, we only include it here. Since we are willing to pass to subsequences, stabilization is trivial in the first case and we will not comment on it further. Property $\FSL$ is really only needed to check $\eqref{item:surface}-\eqref{item:direct_prod}$.

\subsection{Uniform bounds on stable norms} \label{sec:stablenorm} Recall the definition of stable norm in a group:

\begin{dfn} Let $\Gamma$ be a finitely generated group, equipped with a word metric with respect to a finite symmetric generating set $S$. Then for any $\gamma\in\Gamma$, the \emph{stable norm} of $\gamma$ is
	$$\|\gamma\|_\infty := \lim_{n\to\infty} \frac{\|\gamma^n\|}{n}.$$
\end{dfn}
\begin{rmk} Because of submultiplicativity, the stable norm is well-defined and finite (in fact bounded by $\|\gamma\|$).\end{rmk}
\begin{rmk} Stable norm is conjugation invariant.
\label{rmk:stable_class}
\end{rmk}

The remarks leads us to consider the following sets of automorphisms of $\Gamma$.   For any $D>0$ define the set $\calD$ of automorphisms $\psi$ of $\Gamma$ which satisfy 	$$\|\psi(\gamma)\|_\infty\leq D\|\gamma\|_\infty$$
for every $\gamma$ in $\Gamma$.  Since conjugation preserves stable norm, $\calD$ is a union of fibers of the map
$\Aut(\Gamma) \rightarrow \Out(\Gamma)$ and it is more natural to consider the projection of $\calD$ to $\Out(\Gamma)$. We refer to $\calD$ as the set of automorphisms which are {\em $D$--Lipschitz for the stable norm}.

\begin{dfn}
\label{dfn:FSL}  We say $\Gamma$ has $\FSL$ if the projection of $\calD$ to $\Out(\Gamma)$ is finite for all $D>1$.
\end{dfn}

We assume the setting of Theorem \ref{thm:main_tech}. The following proposition relates $\FSL$ to the problem of stabilization.

	\begin{prop} For any $\gamma\in\Gamma$, we have
		\begin{equation*}\|\varphi_n(\gamma)\|_\infty\leq (L+2)\|\gamma\|_\infty.\end{equation*}
	\label{prop:stable_bound}
	\end{prop}
	\begin{proof} Let $\gamma\in\Gamma$. The shortest loop in $M_{t_n}$ representing $\gamma\in \pi_1(M_{t_n},p_0;r)$ has length $ \|\gamma\|+t_n d_M(p_0,\gamma p_0)$ for $t_n$ large enough. Consider the image under $f_n$ of any such loop. Of course, it represents $\varphi_n(\gamma)\in \pi_1(N_{s_n}, f_n(p_0);Lr+C)$ and hence has length at least $\|\varphi_n(\gamma)\|+s_n d_N(f_n(p_0), \varphi_n(\gamma) f_n(p_0))$. On the other hand by the quasi-isometry property, and the fact that the original loop could be represented by roughly $\frac{\|\gamma\|}{r} + \frac{t_n}{r} d_M(p_0,\gamma p_0)$ points where under application by $f_n$, each segment (of length roughly $r$) gets stretched to a segment of length at most $Lr+C$, the image has length at most
	$$(Lr+C)\left(\frac{\|\gamma\|}{r}+\frac{t_n}{r} d_M(p_0,\gamma p_0)\right).$$
Hence we conclude (for $r\geq C$, which we can assume), that
	$$\|\varphi_n(\gamma)\|+s_n d_N(f_n(p_0),\varphi_n(\gamma)f_n(p_0))\leq (L+1)\left(\|\gamma\|+t_n d_M(p_0,\gamma p_0)\right).$$
Now consider the above identity for $\gamma^m$ as $m\to\infty$. If $\gamma$ has infinite order, then for $m\gg1$ we have $\|\gamma^m\|\geq t_n \diam(M)$, so
	\begin{align*}
	\|\varphi_n(\gamma^m)\|	&	\leq 		\|\varphi_n(\gamma^m)\|+s_n d_N(f_n(p_0),\varphi_n(\gamma^m)f_n(p_0))		\\
						&	\leq 		(L+1)\left(\|\gamma^m\|+t_n d_M(p_0,\gamma^m p_0)\right)					\\
						&	\leq 		(L+2)\|\gamma^m\|.
	\end{align*}
Now divide by $m$ on both sides and take the limit to get $\|\varphi_n(\gamma)\|_\infty\leq (L+2)\|\gamma\|_\infty$, as desired.

The above applied for $\gamma$ of infinite order. If $\gamma$ has torsion, then so does $\varphi_n(\gamma)$. Further the stable norm of torsion elements vanishes, so the desired bound is immediate.
\end{proof}

Retaining the terminology of Theorem \ref{thm:main_tech}, the following is now obvious.
\begin{cor}
\label{cor:stablize}
If $\Gamma$ has $\FSL$ then the maps $\phi_n$ stabilize up to conjugacy on a subsequence.
\end{cor}

\subsection{Hyperbolic groups} \label{sec:hyperbolic_stab} In this subsection we prove:
\begin{prop}\label{prop:hyp_FSL}
Hyperbolic groups have Property $\FSL$.
\end{prop}
If stable norm is replaced with conjugacy length (i.e. the minimal word length of a conjugate), then this result is established for torsion-free hyperbolic groups in \cite[Proposition 2.3 ]{MR3382028}. Throughout this subsection, we let $\Gamma$ be a hyperbolic group, generated by a finite set $S$. We denote the Cayley graph of $\Gamma$ with respect to $S$ by Cay$(\Gamma)$, and let $\delta$ be the hyperbolicity constant of Cay$(\Gamma)$. We recall that in hyperbolic groups, stable norm essentially detects conjugacy classes:
\begin{lem}[{see e.g. \cite[Proposition III.F.3.15]{MR1744486}}]\label{lem:dis_tran_length} For every $R>0$, there are only finitely many conjugacy classes with stable norm bounded by $R$. In fact, there is a constant $C,D>0$ depending only  on $\delta$ such that if $\|\gamma\|_\infty<R$ then there is $w\in \Gamma$ in the conjugacy class of $\gamma$ with $|w|<CR+D$.
\end{lem}
We now establish a basic fact about isometries of hyperbolic spaces: displacement of a point $z$ under an isometry is relatively comparable to distance from $z$ to a point with small displacement. This will be helpful in showing existence of fixed points for limit actions on trees.
\begin{lem}\label{lem:find_fix_pt}
Let $s\in \Gamma$ and $x,y\in \text{Cay}(\Gamma)$ such that $d(x,sx)<C$ and $d(y,sy)<d$ for some constants $d, C> \delta$. Then there is a point $z$ on a geodesic segment $[x,y]$ which satisfies $d(z,sz)<10C$ and $d(y,z)<10d$.
\end{lem}
\begin{proof}If $d(x,y)\le 4\min\{d,C\}$ we can choose $z:=x$. Hence we assume that $d(x,y)>4d$ and $d(x,y)>4C$. Pick $z\in [x,y]$ such that $d(z,y)=2d$. Let $w\in [x,sy]$ such that $d(x,w)=d(x,z)=d(x,y)-2d$. By $\delta$-hyperbolicity, $d(z,w)<2\delta$.  Let $t\in [sx,sy]$ such that
	\begin{align*}
	d(sy,t)		=	d(sy,w)	&=	d(x,sy)-(d(x,y)-2d)\\
							&= 	d(x,sy)-d(sx,sy)+2d.
	\end{align*}
By $\delta$-hyperbolicity, $d(t,w)<2\delta$. On the other hand, since $sz\in[sx,sy]$ and $d(sz,sy)=2d$, we get $d(sz,t)<|d(x,sy)-d(sx,sy)|\le C$. By the triangle inequality, we see $d(z,sz)<4\delta + C <10C$.
\end{proof}
We also need the following lemma about isometric actions on trees. The proof can be found in Serre's \cite{MR0607504}. We note that Serre considers simplicial trees but the proof is valid for general $\mathbb{R}$-trees as well.
\begin{lem}\cite[Proposition 26, section I.6]{MR0607504}\label{lem:tree}
Let $\Lambda$ be a group acting on a tree $T$ by isometries and let $s_1,s_2\in\Lambda$ be elements with fixed points. Let $Fix(s_1)$ and $Fix(s_2)$ be the fixed point sets of $s_1$ and $s_2$. If $Fix(s_1)\cap Fix(s_2)=\varnothing$ then $s_1s_2$ does not have fixed point on $T$.
\end{lem}
We will now prove that hyperbolic groups have Property $\FSL$.
\begin{proof}[Proof of Proposition \ref{prop:hyp_FSL}]
Let $\rho_n:\Gamma\to\Gamma$ be automorphisms that are Lipschitz with respect to stable norm. Consider the sequence of action of $\Gamma$ on $\text{Cay}(\Gamma)$ obtained by precomposing the action with $\rho_n$. Define displacement functions $d_n$ by $d_n(x)=\max_{s\in S} d(\rho_n(s)x,x)$. Choose $x_n\in \text{Cay}(\Gamma)$ that minimize $d_n$.

We prove that $d_n(x_n)$ is uniformly bounded. Suppose for the sake of contradiction that $d_n(x_n)$ diverges. Consider the asymptotic cone of $\text{Cay}(\Gamma)$ with respect to the basepoints $(x_n)_n$ and rescaling sequence $(d_n(x_n))_n$. The limit is an $\mathbb{R}$-tree $T$ and the action of $\Gamma$ on $T$ does not have a global fixed point by our choice of scaling sequence and basepoints. On the other hand, by Lemma \ref{lem:dis_tran_length}, for every $s\in S$, and $n\in \mathbb{N}$, there exists $w_n\in \text{Cay}(\Gamma)$ such that $d(\rho_n(s)w_n,w_n)<C$ for some universal constant $C$ because $\rho_n$ is Lipschitz with respect to stable norm.  By Lemma \ref{lem:find_fix_pt}, there exists $w_n'$ on the geodesic segment $[w_n,x_n]$ such that $d(x_n,w_n')<10d_n(x_n)$ and $d(w_n',\rho_n(s)w_n')<10C$. The sequence $(w_n')$ converges to a fixed point of $s$ on $T$. Thus every element of $S$ acts with fixed point on $T$. Moreover, this argument also produces fixed points for the product of any pair of generators.

On the other hand, since $\Gamma$ acts on $T$ without global fixed point, there exist $s_1,s_2\in S$ whose fixed point sets are disjoint. It follows from Lemma \ref{lem:tree} that $s_1s_2$ does not have a fixed point on $T$, which is a contradiction. Therefore, there exist $x_n\in \text{Cay}(\Gamma)$ such that $d_n(x_n)$ is uniformly bounded.

Now pick such a sequence $(x_n)$. Then $d(x_n^{-1}\rho_n(s)x_n,1)$ is uniformly bounded for all $s\in S$ and $n\in \mathbb{N}$. Therefore the $\rho_n$ are all conjugate to a finite collection of homomorphisms and so stabilize on a subsequence.
\end{proof}

\subsection{Free groups} \label{sec:freegp_stab} In this subsection we prove:

	\begin{prop}
	\label{prop:stab_free}
The free group $F_k$ on $k$ generators has Property $\FSL$.
	\end{prop}
We note that free groups are hyperbolic so have property $\FSL$ from the results of the last subsection. We present here a more explicit proof, both because free groups are more essential for applications, and since the basic idea is applied to free products in the next subsection.

We start by recalling that in a free group, word length and stable word length are closely related via cyclically reduced words. Recall that a reduced word $w=s_{i_1}\cdots s_{i_k}$ is cyclically reduced if $s_{i_k}\neq s_{i_1}^{-1}$. This is to say that in $w^2$ there is no cancellation between the last letter of the first copy of $w$ and the first letter of the second copy of $w$.
	\begin{lem} Let $w$ be a cyclically reduced word. Then $\|w\|_\infty = \|w\|$.
	\label{lem:stable_norm_free}
	\end{lem}
	\begin{proof} It is obvious that $\|w^n\|=n\|w\|$ if $w$ is cyclically reduced. This immediately implies the claim.\end{proof}
Now we obtain stabilization for free groups:
	\begin{proof}[Proof of Proposition \ref{prop:stab_free}]  Given an automorphism $\varphi$, change it by conjugation if necessary to ensure that $\|\varphi(s_i)\|$ are minimal (in lexicographical ordering). We are not claiming such a representative is unique, merely take one.  We assume $\varphi$ is $D$--Lipschitz for the stable norm and give
a uniform bound on $\|\varphi(s_i)\|$. This clearly shows $\calD$ projects to a finite set in $\Out(F_k)$.

So let $s$ be any generator. If $\varphi(s)$ is cyclically reduced, we have
		$$\|\varphi(s)\|=\|\varphi(s)\|_\infty \leq D.$$
	If $\varphi(s)$ is not cyclically reduced, we can write (in reduced form)
		$$\varphi(s)=s_i w s_i^{-1}$$
	for some $i$ and word $w$ that does not start with $s_i^{-1}$ and does not end with $s_i$. Now there is some $j$ such that $\varphi(s_j)$ does not start with $s_i$ and does not end with $s_i^{-1}$.
	
	Indeed, if for every $j$, we would have that $\varphi(s_j)$ either started with $s_i$ or ended with $s_i^{-1}$, then consider $\psi:=s_i^{-1} \varphi s_i$. We have, for every $j$:
		$$\| \psi(s_j)\|\leq \|\varphi(s_j)\|$$
	and further $\|\psi(s)\|<\|\varphi(s)\|$. This contradicts minimality of $\varphi$.
	
	Hence we can choose $j$ such that $\varphi(s_j)$ does not start with $s_i$ and does not end with $s_i^{-1}$. Then $\varphi(s) \varphi(s_j)$ is cyclically reduced and hence
		\begin{align*}
		\|\varphi(s)\|\leq \|\varphi(s) \varphi(s_j)\|	&	=		\|\varphi(s)\varphi(s_j)\|_\infty 			\\
												&	\leq 		D \|s s_j\|_\infty 						\\
												&	\leq 		2D.
		\end{align*}
	\end{proof}

\subsection{Free products} \label{sec:freeprod_stab}


In this subsection we show that taking free products with $\bbZ$ preserves $\FSL$ for co-Hopfian groups
$\Delta$.  It is most likely the case that other hypotheses on $\Delta$ suffice as well.  Our motivation
for considering co-Hopfian groups is that higher rank lattices, and particularly all groups known
to have strong property $(T)$, are co-Hopfian.

\begin{prop} Let $\Delta$ be a finitely generated co-Hopfian group without a free cyclic factor. Set $\Gamma=\Delta\ast\bbZ$. Then $\Gamma$ has $\FSL$.\end{prop}

Recall that a group $G$ is co-Hopfian if whenever $H\subseteq G$ is a subgroup with $H\cong G$, then $H=G$.


\begin{proof} Again fix an automorphism $\varphi$ of $\Gamma$ and assume $\varphi$ is $D$--Lipschitz for the stable norm. Again we will, after correcting by $\varphi$ by conjugation, bound the word length of images of generators under $\varphi$.  This suffices to see the projection of $\calD$ to $\Out(\Gamma)$ is finite.  Let $t$ be the generator of the free cyclic factor and consider a generating set $S=S_\Delta\cup\{t, t^{-1}\}$ where $S_\Delta$ is a symmetric generating set for $\Delta$. First note that because $\Delta$ has no free cyclic factor, by the Kurosh subgroup theorem, we can conjugate $\varphi$ so that $\varphi(\Delta)\subseteq \Delta$. Since $\Delta$ is co-Hopfian, it follows that $\varphi(\Delta)=\Delta$.

\begin{claim}If $\varphi:\Delta\ast\bbZ\to \Delta\ast\bbZ$ is an automorphism such that $\varphi(\Delta)=\Delta$ then there are $\delta_1,\delta_2\in\Delta$ such that $\varphi(t)=\delta_1 t^{\pm 1}\delta_2$.\label{cl:auto-formula}
\end{claim}

\begin{proof} We note that for an automorphism $\varphi$ satisfying $\varphi(\Delta)=\Delta$ and $\delta_1\in\Delta$, then $\varphi_1:\Delta\ast\bbZ\to \Delta\ast\bbZ$ defined by $\varphi_1(\delta)=\varphi(\delta)$ for all $\delta\in\Delta$, and $\varphi_1(t)=\delta_1\varphi(t)$ is also an automorphism of $\Delta\ast\bbZ$ satisfying $\varphi_1(\Delta)=\Delta$. Similarly for $\varphi_2$ defined by $\varphi_2|_\Delta=\varphi|_\Delta$ and $\varphi_2(t)=\varphi(t)\delta_2$ for some $\delta_2\in\Delta$. Therefore, we can assume that $\varphi(t)$ is of the form $t^kwt^l$ where $w$ is an element in $\Delta\ast\bbZ$ which does not start and end by nonzero powers of $t$. If $k+l= 0$, then the image of any word containing $t^{\pm1}$ contains both $t^k$ and $t^{-k}$ in its reduced word representation. In particular, $t$ cannot be generated from $\varphi(\Delta\ast\bbZ)$. Thus, $k+l\neq 0$.

We show that $w=1$ and $k+l=\pm1$. Indeed, since $k+l\neq 0$, $w$ cannot be canceled out in image of any reduced word containing $t$ in $\Delta\ast\bbZ$. In particular, $t$ cannot be generated from image of $\varphi(\Delta\ast\bbZ)$, a contradiction. Thus, $w=1$, and it follows that $k+l=\pm1$.\end{proof}

By Claim \ref{cl:auto-formula}, we now conjugate $\varphi$ by elements of $\Delta$ and assume that $\varphi(\Delta)=\Delta$ and $\varphi(t)=t^{\pm1}\delta$, for some $\delta\in\Delta$. Since $\|\varphi(t)\|_\infty=1+\|\delta\|$, we get $\|\delta\| < D$. On the other hand, for $s\in S_\Delta$, $\|\varphi(ts)\|_\infty=1+\|\delta\varphi(s)\|$. It follows that $\|\delta\varphi(s)\|$ is uniformly bounded, hence so is $\|\varphi(s)\|$.\end{proof}

\subsection{Direct products} \label{sec:direct_prod_stab}

Finally we record the case of products such that automorphisms preserve or interchange the factors, and each factor is of the above type. Mild conditions are required, which motivates:

\begin{dfn}
\label{defn:noncommutative}  A group $G$ is called {\em highly non-commutative} if for any nontrivial proper subgroup $H$, the group $HZ(H)$ is also proper where $Z(H)$ is the centralizer of $H$ in $G$.
\end{dfn}

Free groups, hyperbolic groups and irreducible higher rank lattices are all easily to be highly non-commutative. The following is now quite easy

\begin{prop} Let $\Delta_1,\dots,\Delta_k$ be highly non-commutative groups with $\FSL$. Then $\Gamma=\prod_i \Delta_i$ has $FSL$.
 \label{prop:prod_stab} \end{prop}

The proof of the proposition is the following easy exercise: use that the $\Delta_i$ are highly non-commutative
to show that any automorphism of $\Gamma$ permutes the factors. We will only use the following special case.

\begin{cor} \label{cor:prodstab}
Let $G = \prod_I \bbG_i(k_i)$ where $|I| < \infty$, each $k_i$ is a local field and each $\bbG(k_i)$ is a simple algebraic group defined over $k_i$.  Let $\Gamma <G$ be a lattice.  Then $\Gamma$ has Property $\FSL$.
\end{cor}

\begin{rmk}
\label{dumb:remark} We will actually only need Corollary \ref{cor:prodstab} in the special case where $|I|=2$, the lattice $\Gamma=\Gamma_1 \times \Gamma_2$ is reducible and one factor $\Gamma_1$ is a free group, viewed, for example, as a lattice in $SL(2,\bbR)$ or $SL(2,\bbQ_p)$.  To construct superexpanders, the other factor $\Gamma_2$ will be a lattice in a higher rank simple $p$-adic Lie group.  To construct expanders, the choice of the second factor will be almost arbitrary.
 \end{rmk}



\subsection{Geometric cone} \label{sec:geomcone_stab} Finally, we prove stabilization up to conjugacy in the setting of a quasi-isometry of geometric cones, i.e. in the setting of Theorem \ref{thm:main}. This proof is unrelated to the last few sections and to property $\FSL$.

Recall that for $t>0$, we have defined the maps
	$$f_t:M_t\to N_t$$
and found there is $t_0>0$ such that for any $t\geq t_0$, the maps $f_t$ are uniformly coarsely Lipschitz (see Lemma \ref{lem:global_to_local}). We obtain the following description for the induced maps $\Phi_t:=f_{t\ast}$ on coarse fundamental groups. Recall that the coarse fundamental group of a truncated cone is given by the group of lifts of the acting group (see Corollary \ref{cor:coarsepi1_geomcone}). Therefore we can view $\Phi$ and $\Phi_t$ as maps $\widetilde{\Gamma}\to\widetilde{\Lambda}$.

\begin{prop} For $t\gg r\gg 1$, we have $\Phi_t=\Phi$. In particular $\{\varphi_t\}_t$ stabilize.

\label{prop:geomcone_stab}
\end{prop}

\begin{proof} Recall that for $t_0>0$, we have the truncated geometric cone $\GGMtr:=(M\times [t_0,\infty)\times\Gamma)\slash\Gamma$. By Corollary \ref{cor:coarsepi1_geomcone}, for any $t>t_0\gg r\gg 1$, the inclusion map $i_t: M_t\hookrightarrow \GGMtr$ induces an isomorphism on coarse fundamental group. It follows immediately that $\varphi_t:\Gamma\to\Lambda$ is given by the map $\varphi:\Gamma\to\Lambda$ induced by $f$, and hence is independent of $t>t_0$. \end{proof}
	
\section{Identification of growth rate by quasi-isometry}
\label{sec:force_linear}

We will now start the proof of the Main Theorem \ref{thm:main_tech}. Therefore we assume we have $(L,C)$--quasi-isometries $f_n:M_{t_n}\to N_{s_n}$ such that the induced maps $\varphi_n:\Gamma\to\Lambda$ stabilize up to conjugation. In this section, we aim to prove $t_n$ and $s_n$ grow at the same rate:
\begin{prop} Assume the setting of Theorem \ref{thm:main_tech}. Then there exists $K\geq 1$ such that
	$$\frac{1}{K}\leq \frac{s_n}{t_n}\leq K.$$
	\label{prop:same_rate}
\end{prop}
\begin{proof} It suffices to find an upper bound for $\frac{s_n}{t_n}$, since then a lower bound can be obtained by interchanging $M$ and $N$. Fix basepoints $p_0\in M$ and $q_0\in N$. By assumption, there are $\lambda_n\in\Lambda$ such that the morphisms
	$$\lambda_n \varphi_n \lambda_n^{-1}:\Gamma\to\Lambda$$
lie in a finite set. Let $\widehat{\varphi}$ be one of these morphisms and pass to the subsequence such that $\lambda_n \varphi_n \lambda_n^{-1}=\widehat{\varphi}$. Now let $\gamma\in\Gamma$ be a generator and consider the $r$--loop
	$$\widetilde{\jmath}_\Gamma(\gamma)=(p_0,\gamma p_0, p_1,\dots, p_\ell, p_0)$$
where $p_1:=\gamma p_0$ and $(p_1, \dots, p_\ell, p_0)$ is a discretization on the scale $\frac{r}{t_n}$ of a minimizing geodesic from $\gamma p_0$ to $p_0$ in $M$, and such that $d_M(p_i,p_{i+1})=\frac{r}{t_n}$ for every $1\leq i<\ell$. In particular, we have
	\begin{align*}
	d_M(p_0,\gamma p_0)	&=		d_M(p_\ell, p_0)+\sum_{1\leq i<\ell} d_M(p_i, p_{i+1})\\
						&\geq 	(\ell-1)\frac{r}{t_n}.
	\end{align*}
Let $D$ be the maximal displacement of a point under $\gamma$. Rewriting the above inequality, we obtain
	\begin{equation}
	\ell\leq 1+t_n \frac{D}{r}.
	\label{eq:length_bound}
	\end{equation}
Applying $f_n$ to $\jmath_\Gamma$ , we obtain a loop $\beta\in \pi_1(N_{s_n}, f_n(p_0); Lr+C)$ with $Q_\Lambda(\beta)=\varphi_n(\gamma)$. Consider its canonical form (see Definition \ref{dfn:canonical})
	$$\calC(\beta)=\calO(\beta)\ast \calS(\beta).$$
Since the spatial component $\calS(\beta)$ is a path from $\varphi_n(\gamma)f_n(p_0)$ to $p_0$, its length is at least $s_n\delta$, where $\delta$ is the minimal displacement of $\varphi_n(\gamma)$. Note that $\varphi_n(\gamma)$ is conjugate to $\widehat{\varphi}(\gamma)$, so $\delta$ does not depend on $n$. Hence the $(Lr+C)$-path $\calS(\beta)$ consists of at least $\frac{s_n \delta}{Lr+C}$ points.

On the other hand, $\jmath_\Gamma(\gamma)$ consists of $\ell+3$ points, so $\calS(\beta)$ consists of at most $\ell+3$ points. Combining these two bounds on the number of points of $\calS(\beta)$, and also using the bound on $\ell$ from Equation \eqref{eq:length_bound}, we obtain
	$$\frac{s_n\delta}{Lr+C}\leq \ell+3 \leq 4+ t_n \frac{D}{r}$$
and hence
	$$\frac{s_n}{t_n}\leq \frac{4(Lr+C)}{t_n}+\frac{D}{\delta}\left(L+\frac{C}{r}\right)\leq 1+\frac{D}{\delta}\left(L+\frac{C}{r}\right),$$
(if $t_n\geq 4(Lr+C)$). Note that there are only finitely many possibilities for $\delta$, because there are only finitely many options for $\widehat{\varphi}$. Taking the maximum of $1+\frac{D}{\delta}\left(L+\frac{C}{r}\right)$ over all these choices of $\widehat{\varphi}$ yields the desired upper bound. \end{proof}

\section{Construction of a semiconjugacy}
\label{sec:semiconj}

The goal of this section is to construct, under the assumptions of either Theorem \ref{thm:main} or Theorem \ref{thm:main_tech}, a semiconjugacy between the actions of $\Gamma$ on $M$ and $\Lambda$ on $N$. Recall that for $t\gg r\gg 1$, we know that the coarse fundamental group of the level sets $M_t$ (resp. $N_t$) is given by the group of lifts $\widetilde{\Gamma}$ (resp. $\widetilde{\Lambda}$). By assumption, we have maps $f_n : M_{t_n}\to N_{s_n}$ that are uniformly coarsely Lipschitz (in the setting of Theorem \ref{thm:main}) or even uniform quasi-isometries (in the setting of Theorem \ref{thm:main_tech}). In this section, we will only need the uniformly coarsely Lipschitz property.

Further $f_n$ induce morphisms on coarse fundamental group that descend to maps $\varphi_n:\Gamma\to\Lambda$ that stabilize up to conjugation to a map $\widehat{\varphi}:\Gamma\to\Lambda$. In case of the setting of Theorem \ref{thm:main_tech}, we know by the results of the previous section that $t_n \asymp s_n$. By modifying the quasi-isometry constants of $f_n$ a bounded amount, we can assume that $t_n=s_n$, which we will do for the rest of this section. We will refer to the coarse inverse of $f_n$ as $g_n$ and note that we all statements true of $f_n$ are true of $g_n$. In this section we will prove:
	\begin{prop} Assume the setup of Theorem \ref{thm:main} or Theorem \ref{thm:main_tech}. Then there exists a $\widehat{\varphi}$-equivariant Lipschitz map $\widehat{f}:M\to N$.  Moreover there is a Lipschitz map $\widehat{g}:N \rightarrow M$ that we can construct to be $\widehat{\varphi}{\inv}$-equivariant so that the composition $\widehat{g}\circ\widehat{f}:M \rightarrow M$ is $\Gamma$ equivariant.
	\label{prop:semiconj_exist} \end{prop}

Note that we do not claim that $\widehat{f}$ is bi-Lipschitz. However, as noted by interchanging $M$ and $N$, we also obtain an equivariant Lipschitz map $N\to M$. The real content of the last two sentences is that we can arrange for $\widehat{g}\circ\widehat{f}$ to be equivariant for the fixed $\Gamma$ action and not equivariant up to some automorphism of $\Gamma$.  In the next section we will use this to show $\widehat{f}$ is not only invertible and bi-Lipschitz but also affine.	


We provide a brief outline of the proof of Proposition \ref{prop:semiconj_exist}. In Section \ref{sec:no_obstruct} we showed that for $n\gg 1$, the maps $f_n$ respect the data of the extension
	$$1\to \pi_1(M)\to \pi_1(M_{t_n}; r)\to\Gamma.$$
We use this to modify $f_n$ to maps $\widetilde{f}_n$ that are approximately continuous on the scale $\frac{r}{t_n}$ (see Definition \ref{dfn:adjust}). More precisely, images under $\widetilde{f}_n$ of points that are $\frac{r}{t_n}$ close in $M$, are distance at most $\sim \frac{r}{t_n}$ apart (see Lemma \ref{lem:loc_Lip}). We modify these further to account for any inner automorphisms that are needed to obtain stabilization for $\varphi_n$ (which are only assumed to stabilize up to conjugacy), which yields maps $\widehat{f}_n$ that are still approximately continuous. The Arzel\`a-Ascoli theorem allows us to take a limit of the maps $\widehat{f}_n$ to obtain a continuous map $\widehat{f}$. Finally we show this map is equivariant with respect to the stabilization $\widehat{\varphi}:\Gamma\to\Lambda$ of the maps $\varphi_n$.

We will now start by adjusting the coarsely Lipschitz maps $f_n$ so that they are approximately continuous. Recall that for an $r$-path $\alpha$ in $M$, we have defined $Q_\Gamma(\alpha)$ to be the product of the group elements that record local jumps along $\alpha$ (see Definition \ref{dfn:q}).
\begin{dfn} Let $n\gg 1$ such that $Q_\Lambda$ is well-defined and let $p_0\in M$. Define the \emph{global adjustment} $\widetilde{f}_n$ of $f_n$ based at $p_0$ as follows:
Let $\alpha$ be a discretization at the scale $\frac{r}{t_n}$ of a path in $M$ from $p_0$ to $p$. Then set
	$$\widetilde{f}_n(p):=Q_\Lambda( f_n \circ \alpha)^{-1}f_n(p).$$
\label{dfn:adjust}
\end{dfn}
Now we can obtain the promised approximate continuity property of the global adjustment relative to the scale used. We start by showing that if two points in $M$ are nearby, their images under $\widetilde{f}_n$ are close:

\begin{lem} Assume that $t_n>\frac{2(C+1)}{\delta_\Lambda(2C)}$. Then for any $x,y\in M$ such that $d_M(x,y)<\frac{1}{Lt_n}$, we have
	\begin{equation*} d_N(\widetilde{f}_n(x),\widetilde{f}_n(y))\leq \frac{C+1}{t_n}.\end{equation*}
	\label{lem:loc_Lip}
\end{lem}
\begin{proof} Let $\alpha$ be a discretization on the scale $\frac{r}{t_n}$ of a path in $M$ from $p_0$ to $x$. Note that $\alpha\ast(y)$ is an $r$-path from $p_0$ to $y$, where we view $(y)$ as a constant $r$-path. Write $\gamma(x,y):=Q_\Lambda(f_n\circ (x,y))$, where again $(x,y)$ denotes an $r$-path. Then we have
	\begin{align*}
		d_N(\widetilde{f}_n(x),\widetilde{f}_n(y))	
				&=	d_N\left(Q_\Lambda(f_n\circ \alpha)^{-1} f_n(x), Q_\Lambda(f_n\circ ([\alpha],y))^{-1} f_n(y)\right)\\
				&=	d_N\left(Q_\Lambda(f_n\circ \alpha)^{-1} f_n(x), Q_\Lambda(f_n\circ \alpha)^{-1} \gamma(x,y)^{-1} f_n(y)\right)\\
				&=	d_N\left(f_n(x), \gamma(x,y)^{-1} f_n(y)\right)\\
				&=	d_N(\gamma(x,y) f_n(x), f_n(y))\\
				&\leq \frac{C+1}{t_n}.
	\end{align*}
This proves the lemma.\end{proof}
We globalize the local approximate continuity property of the previous lemma:
\begin{lem} $\widetilde{f}_n$ is $(L(C+1),\frac{C+1}{t_n})$-coarsely Lipschitz as a map between the manifolds (with metrics $d_M, d_N$).
\label{lem:approx_cnty}
\end{lem}
\begin{proof} Let $p,q\in M$. Let $c:[0,D]\to M$ be a distance-minimizing unit speed geodesic from $p$ to $q$. We divide $[0,D]$ into segments of length $\frac{1}{Lt_n}$, i.e. set $t_j:=\frac{j}{Lt_n}$ with $0\leq j\leq \ell$ where $\ell=\lfloor LDt_n \rfloor$. Set $x_j:=c(t_j)$. Then
	\begin{equation}
	d_N(\widetilde{f}_n(p),\widetilde{f}_n(q))\leq \sum_{0\leq j < \ell} \left( d_N(\widetilde{f}_n(x_j), \widetilde{f}_n(x_{j+1})) \right) + d_N(\widetilde{f}_n(x_\ell),\widetilde{f}_n(q)).
	\label{dist_est}
	\end{equation}
Since $x_{j+1}\in B_M(x_j ; \frac{1}{Lt_n})$, by Lemma \ref{lem:loc_Lip} we have
	$$d_N(\widetilde{f}_n(x_j), \widetilde{f}_n(x_{j+1}))\leq \frac{C+1}{t_n}.$$
Similarly we have
	$$d_N(\widetilde{f}_n(x_\ell), \widetilde{f}_n(q))\leq \frac{C+1}{t_n}.$$
Using these estimates for every term on the right-hand side of Equation \eqref{dist_est}, and also using $\ell\leq DLt_n$, we obtain
	\begin{align*}
	d_N(\widetilde{f}_n(p),\widetilde{f}_n(q))	&	\leq \ell \frac{C+1}{t_n} + \frac{C+1}{t_n}\\
									&	\leq L(C+1)D+\frac{C+1}{t_n},
	\end{align*}
as desired.
\end{proof}

The maps $\widetilde{f}_n$ are not very nicely related yet, because the maps $\varphi_n:\Gamma\to\Lambda$ do not coincide. However, by assumption $\varphi_n$ stabilize up to conjugation, and we will use this to construct a better sequence of adjustments:

\begin{constr}
\label{constr:equivariant} Because $\varphi_n$ stabilize up to conjugation, there exist $\lambda_n\in\Lambda$ such that $\{c_{\lambda_n}\circ \varphi_n\}$ constitute a finite set of morphisms. Here $c_{\lambda_n}$ denotes conjugation by $\lambda_n$. After passing to a subsequence we can assume that $\widehat{\varphi}:=c_{\lambda_n}\circ\varphi_n$ is independent of $n$. Also define $\widehat{f}_n:M\to N$ by $\widehat{f}_n(p):=\lambda_n \widetilde{f}_n(p)$. \end{constr}

Since $\lambda_n$ acts isometrically on $N$ and $\widetilde{f}_n$ are $(L(C+1),\frac{C+1}{t_n})$-coarsely Lipschitz by Lemma \ref{lem:approx_cnty}, the new adjustments $\widehat{f}_n:=\lambda_n \widetilde{f}_n$ are also $(L(C+1),\frac{C+1}{t_n})$-coarsely Lipschitz. Note in particular that the multiplicative constants are uniformly bounded. Therefore if the additive constants all vanished, we could use the Arzel\`a-Ascoli theorem to obtain a limit map (along a subsequence). Unfortunately the additive constants are in general nonzero -- however, they still tend to 0 as $n\to\infty$. This is enough to push through the proof of the Arzel\`a-Ascoli theorem, and hence there is a subsequence such that
	$$\widehat{f}_{n}\to \widehat{f}$$
uniformly. Further $\widehat{f}$ is $L(C+1)$-Lipschitz. The rest of the section is devoted to proving that $\widehat{f}$ is $\widehat{\varphi}$-equivariant, which completes the proof that $\widehat{f}$ is a semiconjugacy between $\Gamma\curvearrowright M$ and $\Lambda\curvearrowright N$. We start by showing that $\widetilde{f}_n$ and $\widehat{f}_n$ satisfy approximate equivariance properties on the orbit of the basepoint $p_0$:

\begin{lem} Let $\gamma\in \Gamma$. Then for any $n\geq 1$, we have
\begin{align*}
		(i) \qquad 	&	d_N\left(\widetilde{f}_n(\gamma p_0),\varphi_n(\gamma) \widetilde{f}_n(p_0)\right)				\leq \frac{(L+C)\|\gamma\|}{ t_n},\\ 
		(ii)\qquad 	&	d_N\left(\widehat{f}_n(\gamma p_0),\widehat{\varphi}(\gamma) \widehat{f}_n(p_0)\right) \enspace	\leq \frac{(L+C)\|\gamma\|}{ t_n}. 
\end{align*}
\label{lem:bdd_equivar}
\end{lem}

\begin{proof} We prove (i). We obtain two $r$-paths $\alpha$ and $\beta$ in $M_{t_n}$ from $p_0$ to $\gamma p_0$ as follows. Let $\alpha$ be the discretization at the scale $\frac{r}{t_n}$ of a continuous path in $M$ from $p_0$ to $\gamma p_0$, and let $\beta$ be the shortest path in the $\Gamma$--orbit of $p_0$ obtained by writing $\gamma$ as a word in the generators. As usual, $\overline\alpha$ denotes the reverse of $\alpha$. The concatenation $\beta\ast\overline{\alpha}$ is a loop based at $p_0$ representing $\jmath_\Gamma (\gamma)$ in $\pi_1(M, p_0;r)$. Hence we have
\[\widetilde{f}_n(\gamma p_0)	=	Q_\Lambda (f_n \circ \alpha)^{-1} \, f_n(\gamma p_0)	=	Q_\Lambda (f_n \circ \overline\alpha) \, f_n(\gamma p_0),\]
and
\[\varphi_n(\gamma) \, \widetilde{f}_n(p_0)=Q_\Lambda(f_n\circ (\beta*\overline{\alpha})) \, f_t(p_0)=Q_\Lambda(f_n \circ \overline{\alpha}) \, Q_\Lambda (f_n\circ \beta)f_n(p_0).\]
Hence we have
\begin{align*}
d_N\left(\widetilde{f}_n(\gamma p_0) \, , \, \varphi_n(\gamma) \widetilde{f}_n(p_0)\right)		&=				\\
																		&\hspace{-2 cm}	d_N(Q_\Lambda (f_n \circ \overline\alpha) \, f_n(\gamma p_0) \, , \, Q_\Lambda(f_n \circ \overline{\alpha}) \, Q_\Lambda (f_n\circ \beta) \, f_n(p_0))							\\
																		&=				d_N(f_n(\gamma p_0) \, , \, Q_\Lambda(f_n\circ\beta) \, f_n(p_0)).
\end{align*}
We will now bound the latter quantity by exhibiting a path in $N$ from $Q_\Lambda(f_n\circ\beta) \, f_n(p_0)$ to $f_n(\gamma p_0)$ whose length we can bound. Indeed, consider the canonical form 		
	$$\calC(f_n\circ\beta)=\calO(f_n\circ\beta)\ast \calS(f_n\circ\beta)$$
of $f_n\circ\beta$ (see Definition \ref{dfn:canonical}). Recall that $\calS(f_n\circ\beta)$ starts at $Q_\Lambda(f_n\circ\beta) \, f_n(p_0)$, ends at $f_n(\gamma p_0)$, and by Lemma \ref{lem:canonical_length} has length at most
	\begin{align*}
	\ell(\calS(f_n\circ\beta))	&	\leq \ell(f_n\circ\beta)\\
						&	\leq \ell(\beta)(L+C)=(L+C)\|\gamma\|.
	\end{align*}
Therefore
	$$d_N(f_n(\gamma p_0) \, , \, Q_\Lambda(f_n\circ\beta) \, f_n(p_0))\leq \frac{(L+C)\|\gamma\|}{ t_n},$$
which completes the proof of (i). The bound in (ii) follows easily:
\begin{align*}
d_N\left(\widehat{f}_n(\gamma p_0) \, , \, \widehat{\varphi}(\gamma) \, \widehat{f}_n(p_0)\right)
	&	= 		d_N\left(\lambda_n \, \widetilde{f}_n(\gamma p_0) \, , \, (\lambda_n \, \varphi_n(\gamma) \, \lambda_n^{-1}) \, \lambda_n \, \widetilde{f}_n(p_0)\right)\\
	&	=		d_N\left(\widetilde{f}_n(\gamma p_0) \, , \, \varphi_n(\gamma) \, \widetilde{f}_n(p_0)\right)\\
	&	\leq 	\frac{(L+C)\|\gamma\|}{ t_n}.
\end{align*}
\end{proof}

We complete the proof that $\widehat{f}$ is a semiconjugacy between the actions by showing it is $\widehat{\varphi}$-equivariant.

\begin{proof}[Proof of Proposition \ref{prop:semiconj_exist}] By taking the limit as $n\to\infty$ in Lemma \ref{lem:bdd_equivar}.(ii), we see that $\widehat{f}$ is $\widehat{\varphi}$-equivariant on the orbit of $p_0$. Since $\Gamma p_0$ is dense in $M$, it follows immediately that $\widehat{f}$ is $\widehat{\varphi}$-equivariant.

We can construct a map $\widehat{g}'$ simply by applying the same process to the $g_n$ as we did to $f_n$. We can do this for the same subsequence as used in Construction \ref{constr:equivariant}.  We will use $\psi$ to denote the resulting map from $\Lambda \rightarrow \Gamma$ since at this point it is unclear that $\psi$ is $\widehat{\varphi}{\inv}$. By choosing the same subsequence as in Construction \ref{constr:equivariant}, we guarantee that $\widehat{\varphi}\circ \widehat{\phi}$ is an inner automorphism given by $c_{\gamma}$ for some element $\gamma \in \Gamma$. Letting $\widehat{g} = \gamma \circ \widehat{g}'$, we obtain the desired conclusion.  \end{proof}

\section{Promotion of the semiconjugacy}	
\label{sec:conj}

In this section, we will complete the proofs of the Main Theorems \ref{thm:main} and \ref{thm:main_tech}. At this point we have an isomorphism $\widehat{\varphi}:\Gamma\to\Lambda$ and a $\widehat{\varphi}$-equivariant Lipschitz map $\widehat{f}:M\to N$. We will use algebraic methods to promote the map $\widehat{f}$ to an affine conjugacy. Here we say a map is affine if it respects the structure of the homogeneous spaces. More precisely:
\begin{dfn}
\label{definition:affine}  Given homogeneous spaces $G/H$ and $K/C$, a map $f:G/H \rightarrow K/C$ is \emph{affine}
if there exists a homomorphism $\pi: G \rightarrow K$ and an element $k \in K$ such that $f(xH)=k \cdot \pi(x)H$.
If $f$ is surjective, we call the map an \emph{affine submersion}, and if $f$ is a diffeomorphism we call $f$ an \emph{affine isomorphism}.
\end{dfn}
The following result shows that in the setting of homogeneous dynamics, semiconjugacies are affine submersions.
\begin{lem}
\label{lemma:isaffine} Let $G$ and $K$ be compact groups, and let $H<G$ (resp. $C<K$) be a closed subgroup that contains no nontrivial closed normal subgroup of $G$ (resp. $K$). Further let $\Gamma$ (resp. $\Lambda$) be a dense subgroup of $G$ (resp. $K$) and assume $f: G/H \rightarrow K/C$ is a continuous map that is equivariant with respect to an isomorphism $\rho: \Gamma \rightarrow \Lambda$. Then $f$ is an affine submersion.
\end{lem}

\begin{proof}
Since $\Lambda$ is dense in $K$ and $\rho$ is an isomorphism, $f$ is onto.
The map $\rho: \Gamma \rightarrow K$ defines a map $\gr(\rho): \Gamma \rightarrow G \times K$.  Let $L$ be the closure of $\gr(\rho)(\Gamma)$.  Then $L$ projects onto both $G$ and $K$ as $\Gamma$ is dense in $G$ and $\Lambda$ is dense in $K$. We will see that $L$ is the graph of a homomorphism from $G$ to $K$.  Note that we can also define $\gr(f)$ as the subset of $G/H \times K/C$ of the form $(x,f(x))$.  As $f$ is continuous and $\rho$--equivariant, $\gr(f)$ is closed and $\gr(\rho)$--invariant, and therefore also $L$--invariant.  Consider the projections $\pi_1: G\times K \rightarrow G$ and $\pi_2: G \times K \rightarrow K$ and denote their restrictions to $L$ by $\pi_i^L$, $i=1,2$.  Also consider the projections $p_1: G/H \times K/C \rightarrow G/H$ and $p_2: G/H \times K/C \rightarrow K/C$ and notice that $p_1$ is one-to-one when restricted to $\gr(f)$.  This immediately implies that $\ker(\pi_1^L)$ is trivial since $L$ acts effectively on $\gr(f)$ by the assumption on the absence of normal subgroups in $H$ and $C$.  Now we can define a map $\widehat{\rho}:\pi_2 \circ (\pi^L_1){\inv}: G \rightarrow K$ which is a homomorphism extending $\rho$ and whose graph is $L$.  The $\gr(\widehat{\rho})$--invariance of $\gr(f)$ then exactly means that $f$ is $\widehat{\rho}$--equivariant. \end{proof}


The difficulty with Lemma \ref{lemma:isaffine} for applications is that it might be the case that $G=G_1 \times G_2$ and $K=G_1$ and the semiconjugacy is given by projection onto the first factor.  In the setting of our theorems, the fact that the quasi-isometry has a coarse inverse allows us to rule out this kind of behavior because we also have a semiconjugacy in the other direction. In conjunction with the last statement in Proposition \ref{prop:semiconj_exist} the following result, which gives a criterion for promoting a semiconjugacy to an affine conjugacy, finishes the proofs of the Main Theorems \ref{thm:main} and \ref{thm:main_tech}.

\begin{cor}
\label{cor:affineinverse}
Let $G$ and $K$ be compact groups, and let $H<G$ (resp. $C<K$) be a closed subgroup that contains no nontrivial closed normal subgroup of $G$ (resp. $K$). Further let $\Gamma$ (resp. $\Lambda$) be a dense subgroup of $G$ (resp. $K$) and assume there is a isomorphism $\rho: \Gamma \rightarrow \Lambda$ and a continuous $\rho$-equivariant map $f: G/H \rightarrow K/C$ and also a continuous $\rho{\inv}$-equivariant map $h: K/C \rightarrow G/H$.  Then $f$ and $h$ are affine isomorphisms.
\end{cor}

\begin{proof}
We consider the map $f\circ h: G/H \rightarrow G/H$ and  follow the proof of Lemma \ref{lemma:isaffine}.  In this context, the map $\rho$ is the identity map from $\Gamma$ to itself and so $\gr(\rho)$ is the diagonal embedding
of $\Gamma$ in $G \times G$ and the closure $L$ of $\gr(\rho)$ is $G$ embedded diagonally in $G \times G$.  Therefore the map $\widehat{\rho}$ is the identity map on $G$ and this immediately implies $f \circ h$ is an affine isomorphism.  By Lemma \ref{lemma:isaffine}, both $f$ and $h$ are affine submersions.  If the composition of two affine submersions is an affine isomorphism, both maps are affine isomorphisms. \end{proof}

\section{Construction of families of expanders}
\label{sec:app_exp}

In this section we show how to apply Theorem \ref{thm:main_tech} to construct expanders with a wide variety
of distinct geometries.  The main point is to construct large families of group actions that are not conjugate
for groups where we can show quasi-isometry of warped cones implies algebraic conjugacy of actions.

\subsection{Outline} To construct a continuum of non-isomorphic actions with various spectral properties, we will use the following construction.  Let $K$ be a simple compact group and $\Delta < K$ a fixed subgroup.  We will consider a free
group $F_n$ and recall that a generic homomorphism of $F_n$ to $K$ is faithful with dense image.  We then consider {\em left-right actions} $\Delta \times F_n$ on $K$ where $\Delta$ acts on the left and $F_n$ acts on the
right.  We observe that this action can also be written as a left action of $\Delta \times F_n$ on $(K \times K)/\diag(K)$, making it clear that the action fits into our setting of minimal isometric actions on homogeneous
spaces.  We will show that a generic choice of $F_n$ in $K$ gives rise to a free action. Theorem \ref{thm:main_tech} shows that any pair of such actions with quasi-isometric warped cones are conjugate by an affine map.  By making judicious choices of $\Delta$, we can guarantee sufficient
spectral properties so that results of Sawicki and Vigolo allow us to show that the level sets of the cones
are expanders or superexpanders \cite{SawickiGap, Vigolo}.

\subsection{Spectral gaps, (super)expanders, reductions to algebraic constructions}
\label{sec:gap-exp}
In this subsection we recall some definitions concerning spectral gaps, expanders and superexpanders and the
connections between them and warped cones.  At the end of the subsection we reduce the proofs of Theorems \ref{thm:expanders} and \ref{thm:superexpanders} to the existence of continuous families of non-conjugate actions
of well-chosen groups. The more expert reader may choose to skip to Proposition \ref{prop:expanderreduction} and read only reductions which follow easily from known results.

\begin{dfn}
Given a class of Banach spaces $\calE$ of Banach spaces we say that a sequence of graphs $G_n=(V_n,E_n)$ is an expander family for $\calE$ if

\begin{enumerate}
\item  $|V_n| \rightarrow \infty$, and there exists $k>1$ such that each $G_n$ is has degree bounded by $k$, and

\item  given any $E$ in $\calE$  there are $p,\eta >0$ such that for  any family of maps $f_n : V_n \rightarrow E$ we have  $$\frac{1}{|V_n|^2} \sum _{(u,v) \in V_n \times V_n} \|(f_n(u), f_n(v))\|_E^p \leq \frac{\eta}{k|V_n|}\sum _{(u,v) \in E_n} \|(f_n(u), f_n(v))\|_E^p.$$
\end{enumerate}

\end{dfn}

The standard definition of an {\em expander} or {\em expander sequence} is a family of graphs that is an expander for a separable Hilbert space.  It is standard to call a family of graphs a {\em superexpander} if  it is expander for the class $E$ of uniformly convex Banach spaces.  Our results here work equally well if a {\em superexpander} is defined to be an expander of the class $E$ of all Banach spaces of nontrivial type.

Given an action of a group $\Gamma$ on a measure space $(X,\mu)$, one has a natural unitary representation $\rho$ on $L^2(X,\mu)$.  Furthermore, given any Banach space $E$, one can form the space of square-integrable $E$-valued functions $L^2(X,\mu, E)$ and $\Gamma$ has a natural isometric representation $\rho$ on $L^2(X,\mu, E)$ by pre-composition.  The first representation is in fact the case of $E = \bbR$ or $\bbC$.  For $\Gamma$ finitely generated by a set $S$, we say that the $\Gamma$--action on $L^2(X,\mu, E)$ \emph{has a spectral gap} if $$\sup _{\gamma \in S}\|\rho(\gamma)v-v\| \geq \epsilon \|v\|.$$

\begin{prop}
\label{prop:goingup}
If $\Delta < \Delta'$ are finitely generated groups and we have an action of $\Delta'$ on a space $(X,\mu)$ where the restriction of the action to $\Delta$ has a spectral gap on $L^2(X,\mu,E)$, then the $\Delta'$--action does as well.
\end{prop}

In \cite{Vigolo}, Vigolo gives a systematic treatment of how to discretize level sets of warped cones to produce
expander graphs. We call these discretizations the {\em geometric graphs} corresponding to  the $\Gamma$--action on $X$.  We use the following result which combines results of Sawicki and Vigolo, see also papers of Nowak--Sawicki and De Laat--Vigolo \cite{MR3577880, Vigolo, SawickiGap, deLaatVigolo} for related developments.  Special cases of the following proposition were certainly known before these works and the basic ideas seem to go back to the earliest days of the subject.

\begin{prop}
\label{prop:gap}
Let a finitely generated group $\Gamma$ act on a compact metric space $M$ preserving a Radon measure $\mu$ and with a spectral gap on $L^2(X,\mu)$.  Then there exist constants $L>1$ and $C>0$ and graphs $G_n$ such that $G_n$ is $(L,C)$--quasi-isometric to the level set $X_n$ in the warped cone $\calC(\Gamma\curvearrowright X)$ and such that the family of graphs $\{G_n\}$ is an expander sequence.  We call these graphs the {\em geometric graphs} for the action of $\Gamma$ on $M$.

The $\Gamma$--action on $L^2(X,\mu, E)$ has a spectral gap if and only if the {\em geometric graphs} for  the $\Gamma$ action on $M$ are expanders with respect to $E$.
\end{prop}

Finally we remark that if $E$ is a Hilbert space, then it is easy to see that a spectral gap on $L^2(X,\mu)$ is equivalent to one on $L^2(X,\mu,E)$.

To state our key reductions, we need two more definitions.  Given two groups $\Gamma$ and $\Delta$ where $\Gamma, \Delta <K$ we define the {\em left-right} action of $\Gamma \times \Delta$ on $K$ by acting on the left by $\Gamma$ and on the right by $\Delta$.  Given two actions $\rho_1$ and $\rho_2$ of a single group $\Delta$ on a space $X$ we say the actions are {\em conjugate up to automorphism} if there exists an automorphism $\psi: \Delta \rightarrow \Delta$ such that $\rho_1(\delta)=\rho_2(\psi(\delta))$ for any $\delta$ in $\Delta$. This corrects a small discrepancy in standard terminology where a conjugacy between actions of a priori different groups involves a homomorphism identifying the groups but a conjugacy between two actions of the same group implicitly assumes the identification $\psi$ is the identity.

\begin{prop}
\label{prop:expanderreduction}
To prove Theorem \ref{thm:expanders} it suffices to produce a finitely presented dense subgroup $\Gamma$ of a compact
group $K$ with a spectral gap on $L^2(K)$ such that $\Gamma$ has property $\FSL$ and there is a continuum of free left-right actions of $\Gamma \times F_n$ on $K$ which are not conjugate up to automorphism for some $k>1$.
\end{prop}

\begin{prop}
\label{propLsuperexanderreduction}
To prove Theorem \ref{thm:superexpanders} it suffices to produce a finitely presented dense subgroup $\Gamma$ of a compact
group $K$ where $\Gamma$ has strong property $(T)$ and property $\FSL$ such that there is a continuum of left-right actions of $\Gamma \times F_n$ on $K$  which are not conjugate up to automorphism for some $k>1$.
\end{prop}

Both propositions follow immediately from Propositions \ref{prop:gap}, Proposition \ref{prop:goingup}, Theorem \ref{thm:main_FSL} and Corollary \ref{cor:prodstab}.

\subsection{Genericity of freely acting, dense subgroups} \label{sec:free_generic} Note that $K$ as a compact group is also a real algebraic group.  We will write $\Hom(F_n,K)$ for the representation variety of homomorphisms of $F_n$ into $K$ and $\Hom(F_n,K)\sslash K$ for the character variety.  Note that we take the geometric invariant theory quotient here, but that for compact groups this is the same as the set theoretic one. Also observe that $\Hom(F_n, K) \cong K^n$ as varieties and so one can think of $\Hom(F_n,K)\sslash K$ as $K^n\sslash K$ where $K$ action on $K^n$ is by simultaneous conjugation in all coordinates.  Elements of $K^n/K$ are often called {\em multidimensional conjugacy classes} and the action of $K$ on $K^n$ by simultaneous conjugation will be referred to as simply conjugation below.

\begin{lem}\mbox{}
\label{lemma:charactervariety}
\begin{enumerate}
\item  The subset of $\Hom(F_n, K)$ consisting of dense, faithful representations is the complement of a countable union of proper subvarieties.
\item Fix a countable subgroup $\Delta \subseteq K$. The subset of $\Hom(F_n, K)$ such that the left-right action of $\Delta \times F_n$ on $K$ is free is the complement of a countable union of proper subvarieties.
\item Let $K=SU(n)$ and fix a countable dense subgroup $\Delta \subseteq K$.  Then there exists $V \subset K$ that is a countable union of subvarieties such that if $x \notin V$, we have $\langle \Delta, x\rangle\cong \Delta*\bbZ$.
\end{enumerate}
\end{lem}

\begin{proof} We note that $(1)$ is standard and provide a proof for completeness and also to motivate the proof of $(2)$. We note that up to conjugacy, $K$ has only countably many closed subgroups $C$ and any of these is an algebraic subgroup, so that $\Hom(F_n, C)$  is a proper subvariety of $\Hom(F_n,K)$.  Furthermore the set of representations of $F_n$ into any conjugate of $C$ in $K$ can be identified with the $K$--conjugates of $\Hom(F_n, C)$ in $\Hom(F_n,K)$, so the set of homomorphisms of $F_n$ into any conjugate of $C$ is a proper subvariety of $\Hom(F_n,K)\sslash K$.  This shows that the set representations whose image is not dense, is a countable union of closed subvarieties.  We now check that the set of representations that are not faithful is also contained in a (different) countable union of closed subvarieties, which suffices to prove the first claim.  For any word $w\in F_n$, denote by $V(w)\subseteq \Hom(F_n,K)$ the set of representations that vanish on $w$. The word map $\text{ev}_w:\Hom(F_n,K) \rightarrow K$ that evaluates a representation on the word $w$ is clearly a rational map.  It follows that $V(w)$ is a closed subvariety of $\Hom(F_n,K)$. Further note there are only countably many words in $F_n$, so it remains to show that $V(w)$ is a proper subvariety for every $w$.

To do so, it suffices to exhibit a single faithful $F_n$ in $K$, i.e. a single point not in $V(w)$.  To motivate the proof of $(2)$ we proceed differently.  We recall a theorem of Borel that says that the map $\text{ev}_w: \Hom(F_n,K) \rightarrow K$ is not only rational but dominant \cite{MR702738}.  This implies the image is Zariski--dense and in particular that not every point maps to the identity, proving that $V(w)$ is proper.  Since $V(w)$ and its complement are both conjugation invariant, this proves that the image of $V(w)$ in the character variety is also a proper subvariety.

We now analyze stabilizers of the left-right action of $\Delta \times F_n$ on $K$ in order to be able to apply Borel's theorem again.  Computing, we see that $\delta c \gamma{\inv}= c$ if and only if $\delta = c{\inv} \gamma c$. I.e. fixing an element $\delta \in \Delta$, the set of elements $\gamma$ in $F_n$ where the left-right action of $(\delta, \gamma)$ has a fixed point is the set of elements $\gamma$ whose image is a conjugate of $\delta$.  We
fix an enumeration $\{ \delta_i \}_{i \in \bbN}$ of $\Delta$.  For any word $w\in F_n$, the set of representations of $F_n$ into $K$ where the image of $w$ is conjugate to $\delta_i$ is the subset $W(w,i)$ of $\Hom(F_n, K)$ where $\text{ev}_w(k_1, \ldots, k_n)$ is conjugate to $\delta_i$. Recall that conjugacy classes of semisimple elements in algebraic groups are Zariski--closed and that every element of a compact algebraic group is semisimple.
This immediately implies $W(w,i)$ is a closed subvariety since it is the pre-image of a closed subvariety under a rational map.  Since $\text{ev}_w$ is a dominant map to $K$, we also see that $W(w,i)$ is a proper subvariety of $\Hom(F_n, K)$. Since it is clearly conjugation--invariant we are done as above.

We now prove the third point which is quite similar to the first.  We consider a word $w \in \Delta*\bbZ$ as a map
from $K \rightarrow K$ by inserting $x \in K$ into $w$ anytime the generator for $\bbZ$ appears.  We let $V(w)$ be
the set of  $x \in K$ such that $w(x)=\Id$.  This set is clearly a subvariety of $K$ and to prove the result it suffices to see that $V(w)$ is proper for each $w$.  To see this, it suffices to find one faithful representation of
$\Delta*\bbZ$ into $K$.  This follows easily from the existence of one such representation into $G=SL(n,\bbC)$ and the fact that $SU(n)$ is Zariski dense in $SL(n,\bbC)$.  To clarify, we can define a map $\tilde w: G \rightarrow G$ exactly as we define $w$ and let $\tilde V(w)$ be the subset of $G$ where $\tilde w(x)=\Id$.  Since $SU(n)$ is Zariski
dense in $SL(n,C)$ as a complex subvariety and $\tilde V(w)$ is a complex subvariety, if $\tilde V(w)$ is a proper subset of $G$ then $V(w)$ is a proper subset of $K$.  The existence of a faithful embedding of $\Delta *\bbZ$ in $SL(n,\bbC)$ where $\Delta$ embeds in $SL(n,\bbC)$ is originally due to Nisnevich and was rediscovered later by Shalen \cite{MR0004040, MR535185}.
\end{proof}

\noindent{\bf Remark:} Borel's proof that word maps are dominant uses that there are dense free subgroups of at least $\SU(n)$, so our proof of $(1)$ using dominance really depends on the fact that $(1)$ was already known. However, the proof of $(1)$ given above serves well to motivate the proof of $(2)$.

\subsection{Proofs}

Before beginning the proof of Theorem \ref{thm:expanders}, we make a few historical remarks.
We will need a dense subgroup $\Delta <K$ with a spectral gap on $L^2(K)$.   For particular choices of $K$, many of these are constructed as lattices in a Lie group $G$ which have dense embeddings in a compact form $K$ of $G$ by Margulis, Sullivan and Drin'feld \cite{margulisspectral,sullivanspectral, drinfeldspectral}. More recently, examples were found by more elementary methods by Gamburd, Jakobson and Sarnak in \cite{spectralgapconj} where the authors conjecture that a generic dense free subgroup of a compact simple group $K$ will have a spectral gap. In recent work towards that conjecture, Bourgain--Gamburd and Benoist--De Saxce show that any dense free subgroup of $K$ with algebraic entries has a spectral gap  \cite{MR2358056, MR2966656, MR3529116}. If the Gamburd--Jakobson--Sarnak conjecture is correct, one can prove Theorem \ref{thm:expanders} using only left actions of a free group and not left-right actions.

\begin{proof}[Proof of Theorem \ref{thm:expanders} (families of expanders)]
Fix $K=\SU(n)$ and a free group $\Delta\subseteq K$ with a spectral gap on $L^2(K)$. For any $k\in K$, the group $\langle \Delta, k\rangle$ has spectral gap. Further, by Lemma \ref{lemma:charactervariety}.(3), for a generic choice of $k$, we have $\langle \Delta,k\rangle\cong \Delta\ast \bbZ \cong F_n$.  This produces a continuous of free group actions on $K$ which are conjugate if and only if the have representatives in the character variety that are related by an automorphism of the free group.  Since $\Out(F_n)$ orbits in the character variety are countable, this produces continuous families that are not conjugate, so which have warped cones giving rise to quasi-isometrically disjoint families of expanders by Theorem \ref{thm:main_FSL}.

\end{proof}

A different construction is to choose a simple compact Lie group $K$ and a dense subgroup $\Delta \subseteq K$ with a spectral gap on $L^2(K)$. As discussed above, we consider $\Gamma = \Delta \times F_n$ and left-right $\Gamma$--actions on $K$ as defined by an element $\alpha \in \Hom(F_n, K)$. Using Lemma \ref{lemma:charactervariety} we see that off of a countably union of subvarieties in $\Hom(F_n, K)\sslash K$ the resulting left-right action of $\Delta \times F_n$ is free.  This produces a continuumm of non-conjugate $\Delta \times F_n$ actions.  Since the orbit of any action under $\Aut(\Delta \times F_n)$ in $\Hom(F_n, K) \sslash K$ is countable, this suffices to verify the hypotheses of Proposition \ref{prop:expanderreduction}.

\begin{proof}[Proof of Theorem \ref{thm:superexpanders} (families of superexpanders)]
The proof is exactly as the above proof of Theorem \ref{thm:expanders}, though the choices of $\Delta$ and $K$ are more limited.  We can choose $\Delta$ as in the paper of De Laat--Vigolo to be a cocompact lattice in $\SL(d, \bbQ_p)$ that has a dense embedding in $K=\SO(d)$ \cite{deLaatVigolo}. As in that paper it follows from Lafforgue's work on strong property $(T)$ \cite{MR2423763}
that the action on $L^2(K,\mu,E)$ has a spectral gap for any $E$ with nontrivial type and so in particular for any
uniformly convex $E$. More general choices are also possible by work of Liao \cite{MR3190138}. The rest of the proof follows exactly as in the previous case, where we again choose generic homomorphisms $\alpha \in \Hom(F_n, K)$ to define a left-right action of $\Gamma=\Delta \times F_n$ on $K$. \end{proof}

\bibliographystyle{AWBmath}
\bibliography{bibliography}

\end{document}